\documentclass{amsart}
\usepackage[margin=1in]{geometry}
\usepackage[utf8]{inputenc}
\usepackage{amsmath,amsfonts,amssymb,amsthm,tabto,color,commath,multicol,tikz-cd}
\usetikzlibrary{arrows,shapes,automata,backgrounds,petri,fit}
\usepackage{hyperref}
\usepackage{multirow}
\usepackage{amsmath,amssymb,amsthm,amsfonts,latexsym,bbm,xspace,graphicx,float,mathtools,dsfont}
\usepackage{array,xypic}
\usepackage{relsize}
\usepackage{comment}
\usepackage{array} 
\usepackage{xcolor}
\usepackage[all,color]{xy}

\newtheorem{thm}{Theorem}[section]
\newtheorem{lem}[thm]{Lemma}
\newtheorem{prop}[thm]{Proposition}
\newtheorem{coro}[thm]{Corollary}

\theoremstyle{definition}

\newtheorem{ex}[thm]{Example}

\newtheorem{rem}[thm]{Remark}

\newtheorem{defn}[thm]{Definition}

\newcommand{\cA}{\mathcal{A}}
\newcommand{\cB}{\mathcal{B}}
\newcommand{\cE}{\mathcal{E}}
\newcommand{\rE}{\mathrm{E}}
\newcommand{\cF}{\mathcal{F}}
\newcommand{\cX}{\mathcal{X}}
\newcommand{\cP}{\mathcal{P}}
\renewcommand{\cH}{\mathcal{H}}
\newcommand{\cU}{\mathcal{U}}
\renewcommand{\cR}{\mathcal{R}}
\newcommand{\bI}{\mathbb I}
\newcommand{\rI}{\mathrm I}
\newcommand{\cS}{\mathcal S}

\newcommand{\Ksp}{\mathrm{K}_0^{\mathrm{sp}}}
\newcommand{\bZ}{\mathbb{Z}}

\newcommand{\bm}{\mathbf{mult}}

\newcommand{\bdh}{\mathbf{dimh}}

\newcommand{\Emb}{\mathrm{Emb}}

\let\mod\undefined
\DeclareMathOperator{\mod}{\mathrm{mod}}

\DeclareMathOperator{\mods}{\mathrm{mod}}
\DeclareMathOperator{\Hom}{\mathrm{Hom}}
\DeclareMathOperator{\End}{\mathrm{End}}
\DeclareMathOperator{\Ext}{\mathrm{Ext}}
\DeclareMathOperator{\proj}{\mathrm{Proj}}
\DeclareMathOperator{\inj}{\mathrm{inj}}
\let\Im\undefined
\DeclareMathOperator{\Im}{\mathrm{Im}}
\DeclareMathOperator{\Coker}{\mathrm{Coker}}

\DeclareMathOperator{\rk}{\mathrm{rk}}
\DeclareMathOperator{\brk}{\mathbf{rk}}
\DeclareMathOperator{\bdim}{\mathbf{dim}}
\DeclareMathOperator{\Ind}{\mathrm{Ind}}

\title{Invariants of persistence modules defined by order-embeddings}

\author{Claire Amiot, Thomas Brüstle, Eric J. Hanson}

\address{Claire Amiot, Institut Fourier, 100 rue des maths, 38402 Saint Martin d'H\`eres, Université Grenoble Alpes, Institut Universitaire de France}
\email{claire.amiot@univ-grenoble-alpes.fr}

\address{Thomas Brüstle, Départment de Mathématiques, Université de Sherbrooke, Sherbrooke, QC, J1K 2R1, Canada and Bishop’s University, Sherbrooke, QC, J1M 1Z7, Canada}
\email{Thomas.Brustle@USherbrooke.ca, tbruestl@bishops.ca}

\address{Eric J. Hanson, Department of Mathematics, North Carolina State University, Raleigh, NC 27695, USA}
\email{ejhanso3@ncsu.edu}

\begin{document}

\maketitle

\begin{abstract}
      One of the main objectives of topological data analysis is the study of discrete invariants for persistence modules, in particular when dealing with multiparameter persistence modules where complete invariants are unavailable. In many cases, the invariants studied for these non-totally ordered posets $\cP$ can be obtained from restricting a given module $M \in \mod \cP$ to a subposet $\cX$ of $\cP$ that is totally ordered (or more generally, of finite representation type), and then computing the barcode (or the general direct sum decomposition) over $\cX$. 

In this paper, we systematically study invariants arising from order-preserving embeddings of representation-finite posets $\cX$ into $\cP$. Given such an embedding, we associate to a $\cP$-module the multiplicities of indecomposable $\cX$-modules appearing in its restriction, and we analyze the resulting invariants as the embedding varies. This framework recovers classical invariants such as the rank invariant, and allows for a unified treatment of several constructions appearing in the literature.

The restriction functor from $\mod \cP$ to $\mod \cX$ is known to be exact and thus admits both left and right adjoint functors, known as induction and co-induction functors. This allows us to obtain new homological insights, and in particular, we describe bases of the images of these invariants, thereby generalizing the notion of signed barcodes.

We also address the redundancy arising from considering all embeddings of a fixed poset by introducing families of embeddings of increasing size and restricting attention to selected indecomposable modules. This leads to more efficient descriptions of the invariants and yields new bases, including interval-module bases that generalize previously known constructions.
\end{abstract}

\setcounter{tocdepth}{1}
\tableofcontents

\section{Introduction}
\subsection{Setting} 

The study of discrete invariants for persistence modules is one of the main objectives of topological data analysis (TDA). When a pointwise finite-dimensional persistence module $M$ is defined over a totally ordered poset $\cP$, it decomposes into a finite sum of intervals. This data is called the \emph{barcode} of $M$ and completely describes the module up to isomorphism, see \cite[Theorem~1.2]{BCB} and \cite[Section~3.6]{GR}. 
However, when studying multiparameter persistence modules, the underlying poset $\cP$ usually has wild representation type, and thus does not admit such a complete discrete invariant. We refer to  the summary article \cite{BL} for an introduction to multiparameter persistence modules, and why it is relevant in practice to study more general posets than totally ordered ones.

To address the need of TDA for discrete invariants even in the multiparameter case, several invariants have been introduced and studied: the dimension vector $\bdim_\cP$ (aka Hilbert function), the rank invariant $\brk_\cP$, and a number of generalizations thereof, see \cite{BL,BOOS} for an overview of recently studied invariants.

It turns out that, in many cases, these invariants studied for non-totally ordered posets $\cP$ can be obtained from restricting a given module to a subposet $\cX$ of $\cP$ that is totally ordered (or more generally, of finite representation type), and then computing the barcode (or the general direct sum decomposition) over $\cX$. This approach has been described in e.g. 
 \cite{AENY2, DKM,HNOX,La,LW} introducing certain subposets $\cX$ having the form of linear subposets, ``zigzags'', or  ``tours'' and ``courses'', all referring to a poset whose indecomposable representations can be described via intervals.  The full classification of those finite posets whose indecomposable representations can all be described via intervals can be found in \cite{chaptal} and \cite[Section~5]{AET}.

In this paper, we consider general embeddings of  representation-finite subposets $\cX$ into $\cP$ and study invariants obtained by decomposing the restriction of  a given $\cP$-module $M$ to  $\cX$ into its indecompsable summands.
The restriction functor  $\mod \cP \to \mod \cX$ is well-studied, and it is known to be exact and admits both left and right adjoint functors. This allows one to obtain new homological insights, and also to  re-interpret previous results. It also allows one to determine bases, thus generalizing the concept of signed barcodes from \cite{BOO}.

To fix notation, consider an {\em order-embedding} $f: \cX \to \cP$; that is, a poset morphism that is an isomorphism onto its image.
The restriction of a $\cP$-module $M$ to $f(\cX)$ can be described completely by counting the multiplicities of indecomposable direct summands. We assume $\cX$ to be representation-finite, and thus obtain a discrete invariant
  $\bm_{\cX,\cP}^f$ of rank $|\Ind\cX|$, the number of indecomposable $\cX$-modules. 
Repeating this process for all possible embeddings $f: \cX \to \cP$ and collecting the result into a vector indexed over all embeddings, we obtain the invariant  $\bm_{\cX,\cP}$ describing all multiplicities of indecomposable \linebreak $\cX$-modules obtained from all possible ways to restrict the $\cP$-module $M$ to a subposet of form $\cX$. One also defines an equivalent invariant $\bdh_{\cX,\cP}$ by calculating the Hom-dimensions from the set of indecomposable \linebreak $\cX$-modules to the restrictions of the persistence module along all order embeddings. See Definition~\ref{def:embedding} and Theorem~\ref{thm::mult_dimhom_invariants} for details.

Consider for example the poset $\cX_1$ with one element. Then $\cX_1$ has only one indecomposable representation, and its multiplicity in a general module is given by the dimension. Thus listing all possible embeddings of $\cX_1$ into  $\cP$ and counting multiplicities amounts to describing the dimension vector, so $\bm_{\cX_1,\cP}=\bdim_\cP$.
Consider now the poset $\cX_2 =\{1\to 2\}$. This poset is representation-finite, and the multiplicity of the indecomposable representation given by the interval $[1,2]$ can be computed by calculating the rank of the map $M_1 \to M_2$.
This allows one to show that the multiplicity invariant $\bm_{\cX_2,\cP}$ is equivalent to the well-studied rank invariant 
$\brk_\cP$, see Corollary
\ref{corX_2Rank}.
More generally, as discussed in Example  \ref{ex:restrict},
the ``generalized rank invariant'' of \cite{KM} and the ``compressed multiplicities'' of \cite{AENY2} both record the multiplicity of the sincere interval modules $\mathbb{I}_{\cX}$ over certain subsets $\cX \subseteq \cP$ (see also \cite{BDL} for a more general discussion of the relation between these multiplicities and ranks).


\subsection{Results}
For several subposets $\cX$, we address the question of how the multiplicities in $\bm_{\cX,\cP}$ can be computed in terms of linear algebra. See Examples \ref{Ex::mult2} and \ref{Ex::mult4} and the calculations in Section \ref{sec::order-embeddings}. 
In each case, the multiplicities of indecomposable summands of $M$ can be obtained computing ranks or intersections given by the structure maps of the module $M$.
More generally, we show in Theorem \ref{thm::mult_dimhom_invariants}
that the invariant $\bm_{\cX,\cP}$ can be obtained by calculating the Hom-dimensions from the set of indecomposable representations to the given persistence module, thus reducing the problem to solving a system of linear equations. 

Besides computability, the second and equally important quality of an invariant for persistence modules is the stability. For the rank invariant, this question has been addressed  in \cite{La}, and further explored in \cite{BOO,BOOS} where the authors
single out a class of special persistence modules, the ``hook modules'', which are defined as indicator modules of certain shapes called hooks: For $a < b$ in $\cP$, let $[a,b[ \ = \{c \in \mathcal{P}\mid a \leq c \not\geq b\}$. The interval module $\mathbb{I}_{[a,b[}$ is called a \emph{hook module}. Note that $\mathbb{I}_{[a,b[}$ is the cokernel of the natural map $\mathbb{I}_{[b,\infty)} \rightarrow \mathbb{I}_{[a,\infty)}$ between indecomposable projective $\cP$-modules.
It is shown in \cite{BOO} that the hook modules play the role of projective objects in a relative homology theory given by some exact structure $\mathcal{E}$, thus every persistence module $M$ admits a resolution by hook modules. One can then show that these hooks form a basis for the (image of the) rank invariant (see Definition \ref{def::basis}), which in turn allows one to define a ``signed barcode'', leading to a stability result for the considered invariant \cite{BOOS}.

The viewpoint of relative homological algebra has further been extended in many recent works (see e.g. \cite{AENY,BBH,BBH2,BOO,BOOS,KTH,OS}), but we are only now able to conceptually explain the appearance and importance of hook modules and their generalizations:
Studying embeddings  $f: \cX \to \cP$, the corresponding restriction functor $ f^* :\mod \cP \to \mod \cX$ admits both a left adjoint $f_! : \mod \cX \to \mod \cP$ (called coinduction) and a right adjoint $f_* : \mod \cX \to \mod \cP$ (called induction).
One can then use
the adjunction formula to show that the set 
$$\cH = \{ f_!U \mid U \mbox{ indecomposable } \cX\mbox{-module and } f:\cX \to \cP  \mbox{ an embedding} \}$$
forms exactly the indecomposable relative projective objects for the exact structure defined as the class $\cE$ of short exact sequences
    $0 \rightarrow L \rightarrow M \rightarrow N \rightarrow 0$
    in $\mods \cP$ which have the property that
    $$0 \rightarrow f^*L \rightarrow f^* M \rightarrow f^* N \rightarrow 0$$
    splits for all embeddings $f: \cX \to \cP$. See Proposition \ref{prop::relative-proj}.
Coming back to the rank invariant, we study embeddings of the poset $\cX_2$, and since every indecomposable $\cX_2$-module is the cokernel of a map between indecomposable projectives, we see easily that the set $\cH$ above is formed exactly by the hook modules in this case.

Moving on to embeddings of more general posets, we consider $\cX=\cX'_3$ or $\cX''_3$:

\[
  \scalebox{0.8}{
  \begin{tikzpicture}

\begin{scope}[xshift=4cm]

\node(x) at (0,0) {$\cX'_3=$};
\node (A) at (1,0) {$1$};
\node (B) at (2,0.5) {$2$};
\node (C) at (2,-0.5) {$3$};
\draw[->] (A)--(B);
\draw[->] (A)--(C);
\end{scope}

\begin{scope}[xshift=8cm]
\node(x) at (0,0) {$\cX''_3=$};
\node (A) at (1,0.5) {$1$};
\node (B) at (1,-0.5) {$2$};
\node (C) at (2,0) {$3$};
\draw[->] (A)--(C);
\draw[->] (B)--(C);
\end{scope}
 \end{tikzpicture}}\]
We notice that the corresponding generalized hook modules are not always given just by specifying a subposet of $\cP$, but they are given by vector spaces of dimension two at some vertices, see Example 
\ref{ex::f!U modules}. These can be illustrated  for the $n \times m$ grid  (the poset $\{1,2, \ldots ,n\}\times \{1,2, \ldots ,m\}$ with the usual product order) as follows: 
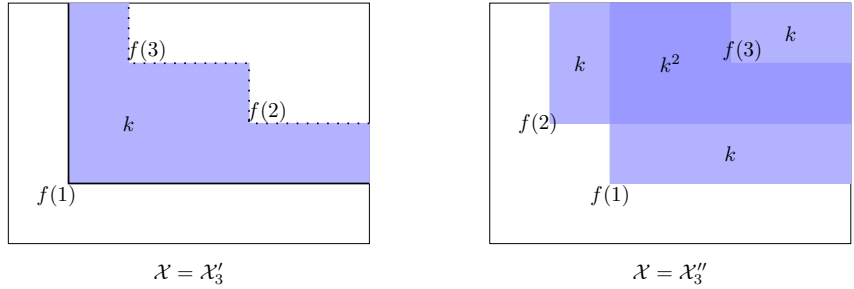
\[ \scalebox{0.8}{
  \begin{tikzpicture}

 \begin{scope}[xshift=-8cm]

\draw (0,0)--(6,0)--(6,4)--(0,4)--(0,0);

\draw[blue!30, fill=blue!30] (1,1)--(6,1)--(6,2)--(4,2)--(4,3)--(2,3)--(2,4)--(1,4)--(1,1);

\draw[thick] (1,4)--(1,1)--(6,1);
\draw[thick, loosely dotted] (6,2)--(4,2)--(4,3)--(2,3)--(2,4);
\node at (0.8,0.8) {$f(1)$};
\node at (4.3,2.2) {$f(2)$};
\node at (2.3,3.2) {$f(3)$};
\node at (2,2) {$k$};

\node at (3,-0.5) {$\cX=\cX'_3$};

\end{scope}

\draw (0,0)--(6,0)--(6,4)--(0,4)--(0,0);

\draw[blue!30, fill=blue!30] (1,2)--(2,2)--(2,4)--(1,4)--(1,2);

\draw[blue!30, fill=blue!30] (2,1)--(6,1)--(6,2)--(2,2)--(2,1);

\draw[blue!30, fill=blue!30] (4,3)--(6,3)--(6,4)--(4,4)--(4,3);

\draw[blue!40, fill=blue!40] (2,2)--(6,2)--(6,3)--(4,3)--(4,4)--(2,4)--(2,2);

\node at (4.2,3.2) {$f(3)$};
\node at (0.8,2) {$f(2)$};
\node at (2,0.8) {$f(1)$};
\node at (1.5,3) {$k$};
\node at (4,1.5) {$k$};
\node at (5,3.5) {$k$};
\node at (3,3) {$k^2$};

\node at (3,-0.5) {$\cX=\cX''_3$};

\end{tikzpicture}}\]
Already in \cite{BOO}, they specified a second basis for the rank invariant, given by rectangle modules $\mathbb I_{ [a , b]}$ with support $[a,b]$, which are somewhat easier to deal with than the hook modules  (note that one needs to consider right-open rectangles when the poset $\cP$ is infinite). We aim to generalise this result here in terms of general embeddings:
Given  a poset embedding $f: \cX \rightarrow \cP$ and $U \in \mods \cX$, the adjunction formulas yield a natural map $\theta_fU: f_! U \rightarrow f_*U$. We denote  $\Theta_fU = \mathrm{Im}(\theta_fU)$.
   (This is the so-called ``intermediate extension functor'', see \cite[Section~4.2]{AET}  for further details and historical context.)
    One easily verifies that in case $\cX=\cX_2$, the set 
$$\cR = \{ \Theta_fU \mid U \mbox{ indecomposable } \cX\mbox{-module and } f:\cX \to \cP  \mbox{ an embedding } \}$$ 
coincides with the set of rectangle modules.
However, before moving on to find bases in a more general setting, we aim to reduce redundancy. Considering sets like 
$$\cH = \{ f_!U \mid U \mbox{ indecomposable } \cX\mbox{-module and } f:\cX \to \cP  \mbox{ an embedding} \}$$ 
produces the same $\cP$-modules several times from different embeddings $f$ and different indecomposables $U$,  for instance the simple modules are double counted whenever the images of two different embeddings $f_1,f_2:\cX \to \cP$ intersect. It is therefore more efficient to first collect all simple $\cP$-modules (once) via embeddings of the one-point poset $\cX_1$, then consider all two-dimensional $\cP$-modules obtained from the interval $[1,2]$ by embedding the poset $\cX_2 =\{1\to 2\}$ into $\cP$ and so on.
In other words, instead of studying the set of all indecomposables obtained from embeddings of one poset $\cX$, we study now embeddings of several posets of increasing sizes, but limit attention to only some indecomposables (that have not been obtained from embedding of smaller posets previously).
The setting we consider here is where we fix a family of posets  $\cF=\{\cX_i\}$ together with certain embeddings into $\cP$, and then we define $\bm_{\cF}$ to count the multiplicity of the characteristic indecomposable module $\mathbb{I}_{\cX_i}$ in the restriction of $M$ to $\cX_i$ along the various embeddings. So only one indecomposable is considered for each poset embedding.
As discussed in Example~\ref{ex:iterated_mult}, both the generalized rank invariant of \cite{KM} and the various compressed multiplicities of \cite{AENY2} can be realized in the form $\bm_\cF$. 

As explained in Section~\ref{sec:families},
there are examples of the invariants $\bm_{\cX,\cP}$ which are equivalent to invariants of the form $\bm_\cF$. In such cases, the perspective of iterated embeddings is useful for removing some of the redundant information stored in the invariant $\bm_{\cX,\cP}$. Aside from this decrease in redundancy, this change in perspective also allows us to describe explicit bases of some of the invariants $\bm_{\cX,\cP}$. See Theorem~\ref{thm::IntervalBasis}.

With  $\cF:=\{\cX_1,\cX_2\}$ for $\cX_1$ and $\cX_2$ the one- and two-point poset as before, we can recover the result that the set $\cR$ (the rectangles in this case)  forms a basis for $\bm_\cF$.

\begin{thm} Let $\cP$ be a finite connected poset, and let $\cX$ be one of the following posets

\[
  \scalebox{0.8}{
  \begin{tikzpicture}

\begin{scope}[xshift=4cm]

\node(x) at (0,0) {$\cX'_3=$};
\node (A) at (1,0) {$1$};
\node (B) at (2,0.5) {$2$};
\node (C) at (2,-0.5) {$3$};
\draw[->] (A)--(B);
\draw[->] (A)--(C);
\end{scope}

\begin{scope}[xshift=8cm]
\node(x) at (0,0) {$\cX''_3=$};
\node (A) at (1,0.5) {$1$};
\node (B) at (1,-0.5) {$2$};
\node (C) at (2,0) {$3$};
\draw[->] (A)--(C);
\draw[->] (B)--(C);
\end{scope}
 \end{tikzpicture}}\]
Assume that 
for any $a<b$ in $\cP$ there exists $f\in \Emb_{\cX}^\cP$ with $\{a,b\}\in \Im f$, then the set 
$$\cR = \{\Theta_f(U) \mid (f,U)\in \Emb_{\cX}^\cP \times \Ind \cX\}$$
forms a basis of $\bm_{\cX,\cP}$.

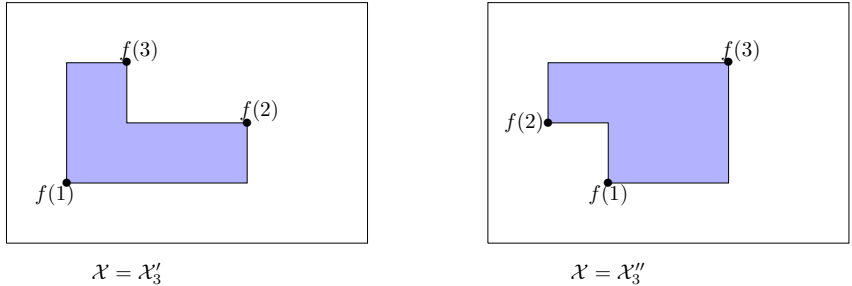
 \[
  \scalebox{0.8}{
  \begin{tikzpicture}

\draw (0,0)--(6,0)--(6,4)--(0,4)--(0,0);

\draw[fill=blue!30] (1,1)--(4,1)--(4,2)--(2,2)--(2,3)--(1,3)--(1,1);

\node at (1,1) {$\bullet$};
\node at (0.8,0.8) {$f(1)$};
\node at (4,2) {$\bullet$};
\node at (4.2,2.2) {$f(2)$};
\node at (2,3) {$\bullet$};
\node at (1.6,1.6) {$k$};
\node at (2.2,3.2) {$f(3)$};
\node at (2,-0.5) {$\cX=\cX'_3$
};

\begin{scope}[xshift=8cm]
\draw (0,0)--(6,0)--(6,4)--(0,4)--(0,0);

\draw[fill=blue!30] (1,2)--(2,2)--(2,1)--(4,1)--(4,3)--(1,3)--(1,2);

\node at (4.2,3.2) {$f(3)$};
\node at (0.6,2) {$f(2)$};
\node at (2,0.8) {$f(1)$};
\node at (1,2) {$\bullet$};
\node at (2,1) {$\bullet$};
\node at (4,3) {$\bullet$};
\node at (3,2) {$k$};
\node at (2,-0.5) {$\cX=\cX''_3$};
\end{scope}
\end{tikzpicture}}\]
 We note that the  generalized rectangle modules described in $\cR$ are always given by specifying a subposet of $\cP$, with all vector spaces of dimension zero or one.
\end{thm}

\subsection{Acknowledgements}
Claire Amiot is supported by the ANR project CHARMS and by the Institut Universitaire de France.
Thomas Br\"ustle was supported by the CRM (Centre de recherches math\'ematiques, Montreal) and Eric Hanson by IUF to visit Institut Fourier, Grenoble. Most of this project was realised during these stays. Thomas Brüstle and Eric Hanson were partially supported by NSERC Discovery Grant RGPIN/04465-2019. Eric Hanson was also partially supported by the Canada Research Chairs program CRC-2021-00120, NSERC Discovery Grant RGPIN-2022-03960, and an AMS-Simons travel grant. A portion of this work was completed while Eric Hanson was a postdoctoral student at l'Université de Sherbrooke. The authors also thank Benjamin Blanchette, Steffen Oppermann, Steve Oudot, and Luis Scoccola for stimulating discussions on the subject.

\section{$\bZ$-linear invariants }\label{section::invariants}

For a finite poset $\cP$,
 we study the category $\mod \cP$ as the category of  functors $\cP\to \mod k$ for some fixed field $k$. It is a $k$-linear abelian category with finite dimensional $\Hom$-spaces which satisfies the Krull-Schmidt property; that is, every object is isomorphic to a finite direct sum of indecomposable objects, and this decomposition is unique up to permutation. 
 
An object $M\in \mod \cP$ is also called a $\cP$-module.  We denote by $M_a$ the vector space $M(a)$, and by $M_{a\to b}$ the element in $\Hom_k(M_a,M_b)$ given by the functor $M$ applied to the pair $a\leq b\in \cP$ .
We further denote  $\bdim_a M :=\dim_k M_a$, and  $\brk_{a,b} M:=\rk M_{a\to b} $.

The \emph{split Grothendieck group $\Ksp(\cP)$  of $\mod \cP$} is the free abelian group generated by isoclasses $[M]$ of objects $M$ in $\mod \cP$ modulo the relation $[M\oplus N]=[M]+[N]$ for any objects $M,N$ in $\mod \cP$. 
Fixing  a set $\Ind \cP$ of representatives of isoclasses of the indecomposable $\cP$-modules, the set $\{[U]\in \Ksp(\cP) \mid U\in \Ind \cP\}$ forms a basis for $\Ksp (\cP)$.  In other words, decomposing $M$ into indecomposable direct summands $M\simeq \bigoplus_{U\in \Ind\cP}U^{\alpha_U}$, the vector  $(\alpha_U)_{U\in \Ind\cP}$ uniquely determines $M$  up to isomorphism, by the Krull-Schmidt property.

We say that a poset $\cP$ is of \emph{finite representation type (or representation-finite)} if $\Ind \cP$ is finite. Therefore the free abelian group $\Ksp (\cP)$ is finitely generated if and only if $\cP$ is of finite representation type. 

Some indecomposable $\cP$-modules have been object of intense studies in the persistence theory literature. One example of such modules are the {\em rectangle modules} $\mathbb I_{ [a , b]}$
(with support $[a,b]$), see e.g. \cite{BLO,KM}: 
Every pair $a \le b$ in $\cP$ yields an indecomposable $\cP$-module  $U=\mathbb I_{ [a , b]}$ with  $U_x = k$ for all $a \le x \le b$ in $\cP$ and $U_{x\to y}$ the identity map for all $a \le x \le y \le b$ in $\cP$, while all other vector spaces and linear maps are zero. 

More generally, a subset $\cS \subseteq \cP$ is called an \emph{interval} if $S$ is order convex (in the sense that any $a \leq b \leq c$ in $\mathcal{P}$ with $a, c \in S$ also has $b \in S$) and connected (in the sense that any $a, b \in \cS$ are related by an alternating chain $a = s_0 \leq s_1 \geq \cdots \geq s_k = b$ with each $s_i \in \cS$). Each interval determines an \emph{interval module}\footnote{Note that interval modules are also called ``spread modules'' in some works, e.g. \cite{BBH,KTH}.} $\mathbb{I}_{\cS}$ analogous to the rectangle module $\mathbb{I}_{[a,b]}$. Note that interval modules are always indecomposable, see e.g. \cite[Proposition~2.2]{BL2}. These modules have featured prominently in many recent works, see e.g. Sections 9 and 10 of the survey article \cite{BL} and the references therein. We will consider the following special classes of interval modules in this paper.

\begin{ex}\label{ex:intervals}
    \begin{enumerate}
        \item For $a \in \mathcal{P}$, let $[a,\infty) = \{b \in \mathcal{P} \mid a \leq b\}$. The interval module $P_a:= \mathbb{I}_{[a,\infty)}$ is an indecomposable projective, and all indecomposable projective $\mathcal{P}$-modules have this form.
        \item For $a \in \mathcal{P}$, let $(-\infty,a] = \{b \in \mathcal{P} \mid a \geq b\}$. The interval module $I_a:=\mathbb{I}_{(-\infty,a]}$ is an indecomposable injective. Note also that all indecomposable injective $\mathcal{P}$-modules have this form.
        \item For $a \in \mathcal{P}$, the interval module $S_a:= \mathbb{I}_{\{a\}}$ (which is also a rectangle module since $\{a\} = [a,a]$) is the simple top of $P_a$ and the simple socle of $I_a$. Note that all simple $\mathcal{P}$-modules have this form.
        \item For $a < b$, let $[a,b[ \ = \{c \in \mathcal{P}\mid a \leq c \not\geq b\}$. The interval module $H_{a,b}:= \mathbb{I}_{[a,b[}$ is called a \emph{hook module}, see \cite{BOO}. Note that $H_{a,b}$ is the cokernel of the natural injective map $P_b \rightarrow P_a$.
        \item For $a < b$, let $]a,b] \ = \{c \in \mathcal{P} \mid b \geq c \not\leq a\}$. The interval module $C_{a,b}:=\mathbb{I}_{]a,b]}$ is called a \emph{cohook} module. Note that $C_{a,b}$ is the kernel of the natural surjective map $I_a \rightarrow I_b$.
        \item If $\cP$ is connected, then it admits a unique sincere interval module $\mathbb{I}_\cP$. (Recall that \emph{sincere} means that $(\mathbb{I}_\cP)_a \neq 0$ for all $a \in \cP$.)
        
    \end{enumerate}
\end{ex}

\subsection{Definitions and first examples}

\begin{defn}
A \emph{$\bZ$-invariant} (or \emph{invariant} for short) on $\mod \cP$ is a $\mathbb Z$-linear map $\Phi:\Ksp(\cP)\longrightarrow \bZ^{I}$ for some set $I$. In other words, $\Phi$ can be viewed as a map $\mod \cP\to \bZ^{I}$ which is constant on isoclasses of modules and such that $\Phi(M\oplus N)=\Phi(M)+\Phi(N)$.  An invariant $\Phi$ is \emph{finite} if its image is finitely generated. We denote by $\rk \Phi$ the rank of its image. 
 We say that $\Phi:\Ksp(\cP)\longrightarrow \bZ^{I}$ is a {\em complete} invariant if 
 $\Phi$ is injective. 
\end{defn}

\begin{ex}\label{example:invariants}\,
\begin{enumerate}
\item Counting the multiplicities of all indecomposable direct summands yields an invariant  $\bm_\cP: \Ksp(\cP)\to \bZ^{\Ind \cP}$ of rank $|\Ind\cP|$. It is defined by $\bm_\cP(M)_U=\alpha_U$ where $M\simeq \bigoplus_{U\in \Ind\cP}U^{\alpha_U}$ is the decomposition of $M$ into indecomposable objects. If $\cP$ is a poset of finite representation type, then $\bm_\cP$ is a finite invariant. 
More generally, if $\cU\subseteq \Ind\cP$ is a finite family of indecomposable modules, we define $\bm_{\cP}^\cU:\Ksp(\cP)\to \bZ^{\cU}$ by composing $\bm_\cP$ with the natural projection of $\bZ^{\cU} \to \bZ^{\Ind \cP}$.

For instance, if $\cU = \{\mathbb I_{ [a , b]} \mid a \le b\}$ is the set of all rectangle modules then a module $M$ is called {\em rectangle-decomposable} if $\bm_\cP(M)_U= 0 $ for $U \not\in \cU$, or equivalently when $\bm_{\cP}(M)$ is obtained from $\bm_{\cP}^\cU(M)$ by the natural inclusion of $\bZ^{\cU} \to \bZ^{\Ind \cP}$.

\item 
For $a\in \cP$, $\bdim_a:\Ksp (\cP)\to \bZ$ is an invariant of rank $1$  given by $\bdim_a(M) = \dim_k M_a$.  The map $\bdim_\cP:=\bigoplus _{a\in\cP}\bdim_a$ is an invariant of rank $|\cP|$.  This invariant is often known as the \emph{dimension vector} or \emph{Hilbert function}.

\item 
Denote the set of ordered pairs in $\cP$ by $$\{\leq_\cP\}:=\{(a,b)\in \cP^2 \mid a\leq b\}.$$ For $(a,b)\in  \{\leq_\cP\}$, the map $\brk_{a,b}:\Ksp(\cP)\to \bZ$  given by $\brk_{a,b}(M) = \rk M_{a\to b}$ is an invariant of rank $1$. The map $\brk_\cP:=\bigoplus_{(a,b)\in \{\leq_\cP\}}\brk_{a,b}$ is an invariant of rank $|\{\leq_\cP\}|$ called the \emph{rank invariant} (see \cite{CZ,BOO}). 

\item For any finite family $\cU\subseteq \Ind\cP$, the map $\bdh_{\cP}^{\cU}:\Ksp(\cP)\to \bZ^{\cU}$ defined by $\left(\bdh_{\cP}^{\cU}(M)\right)_U:=\dim_k \Hom_\cP(U,M)$ for $U\in\cU$ is an invariant of rank $|\cU|$. Following \cite{BBH},  invariants of this form are called  \emph{dim-Hom invariants}. 

\end{enumerate}

\end{ex}

\subsection{Comparing invariants}

Of course, a finite invariant can only be  complete when the poset $\cP$ is representation-finite. One goal of persistence theory is to extract useful information from $M \in \mod \cP$ by some finite invariants even when the poset $\cP$ is  not representation-finite. The following definitions provide tools to compare different invariants.

\begin{defn}
Two invariants $\Phi_1,\Phi_2$ on $\mod \cP$ are \emph{equivalent} (in symbols $\Phi_1\simeq \Phi_2$) if there exists an isomorphism $\varphi: \Im \Phi_1\to \Im \Phi_2$ such that $\varphi\circ \Phi_1=\Phi_2$. 
\end{defn}

\begin{defn}
Let $\Phi_1,\Phi_2$ be two  invariants on $\mod \cP$. We say that $\Phi_1$ is \emph{finer} than $\Phi_2$, and denote it by $\Phi_1\geq \Phi_2$, if there exists a $\bZ$-linear map $\varphi: \Im \Phi_1\to \Im \Phi_2$ such that $\varphi\circ \Phi_1=\Phi_2$ (in particular, $\varphi$ is surjective). This defines a partial order on the equivalence classes of invariants on $\mod \cP$. 
\end{defn}

 \begin{ex}\label{ex:compare_invariants}\,
\begin{enumerate}

\item Let $\cP$ be a poset of finite representation type. Then the invariant $\bm_\cP$ is complete, and thus $\bm_\cP$ is finer than any other $\bZ$-invariant.

\item 
The diagonal inclusion $\cP\to \{\leq_\cP\}$ induces an  injective map $\varphi:\bZ^{\cP}\to \bZ^{\{\leq_\cP\}}$  whose transpose $\varphi^T:\bZ^{\{\leq_\cP\}}\to \bZ^{\cP}$ is surjective and satisfies $\varphi^T\circ \brk_\cP=\bdim_\cP$, hence $\brk_\cP\geq \bdim_\cP$. In general, the rank invariant is strictly finer than the dimension vector, in fact for a connected poset $\cP$ the rank invariant is equivalent to the dimension vector if and only if $|\cP|=1$. 
It is a classical result in representation theory that the rank invariant $\brk_\cP$ is complete for a totally ordered poset $\cP$.

\item Let ${\rm Proj}  \; \cP \subseteq \Ind \cP$ the subset of indecomposable projective modules (see Example~\ref{ex:intervals}(1)). 
It is classical that $\bdim_\cP=\bdh^{\rm Proj \cP}_\cP$, so these invariants are equivalent. See \cite[Proposition 5.1]{BBH}.

\item Let $\cH\subset \Ind \cP$ be the subset formed by the hook modules and the projective indecomposable modules (see Example~\ref{ex:intervals}). It is shown in \cite{BOO} and \cite{BBH} that the invariants $\brk_\cP$ and $\bdh_{\cP}^{\cH}$ are equivalent. 
\end{enumerate}

\end{ex}

\begin{prop}\label{prop::mult-dimhom}
Let $\cP$ be a poset of finite representation type. Then the invariants $\bm_\cP$ and $\bdh_{\cP}^{\Ind \cP}$ are equivalent.
\end{prop}

\begin{proof}
Denote by $\{U_1,\ldots, U_\ell\}=\Ind \cP$. For $1\leq i,j\leq \ell$ we have $\bm_{U_i}(U_j)=\delta_{i,j}$ and $\bdh_{U_i}(U_j)=\dim_k\Hom_{\cP}(U_i,U_j)$. Hence if we denote by $\Phi\in {\rm Mat}_\ell(\bZ)$ the matrix defined by $\Phi_{i,j}:=\dim_k\Hom_{\cP}(U_i,U_j)$, we have that $\Phi\circ \bm_\cP=\bdh_{\cP}^{\Ind \cP}$ showing that $\bm_\cP\geq \bdh_{\cP}^{\Ind \cP}$.

Since $\cP$ is representation-finite, it is 
 representation-directed
by \cite[Theorem~1.4]{simson}. This means that the transitive closure of the relation $U_i \leq U_j$ whenever $\Hom(U_i,U_j) \neq 0$ is a partial order on $\Ind \cP$.  This in turn implies that every indecomposable module admits only the trivial endomorphism, thus it is a brick.
 We conclude that $\Phi$ is an isomorphism using the following lemma with $Y= \Ind \cP$ and $\varphi_i(j)=\Phi_{i,j}.$
\end{proof}

\begin{lem}\label{lemma:basisPoset}
Let $Y$ be a finite poset. For each $y\in Y$ assume that we have $\varphi_y\in \bZ^{Y}$ satisfying 
\begin{itemize}
\item $\varphi_y(y)=1$
\item $\varphi_y(x)\neq 0 \Rightarrow y\geq x$ 
\end{itemize}
Then the family $(\varphi_y, y\in Y)$ is a basis of the abelian group $\bZ^Y$.
\end{lem}

\begin{proof}
 Label the elements $\{y_1,\ldots, y_n\}$  of $Y$ in such a way that $y_i\leq y_j$ implies $i\leq j$. Consider the dual basis $(\delta_{y_i})_{i=1,\ldots, n}$ of $\bZ^Y$ defined by $\delta_{y_i}(y_j)=\delta_{i,j}$. Then the matrix expressing the $(\phi_{y_i})_{i=1,\ldots,n}$ in the basis $(\delta_{y_i})_{i=1,\ldots,n}$ is in ${\rm GL}_n(\bZ)$ since it is an upper triangular matrix with $1$ on the diagonal. Hence the family $(\varphi_{y_i})_{i=1,\ldots,n}$ is a $\bZ$-basis. 
\end{proof}

\section{Invariants defined by order-embeddings}\label{sec::order-embeddings}

\begin{defn}
Let $\cX$ and $\cP$ be finite posets. A map $f:\cX\to \cP$ is an \emph{order-embedding} if $x\leq x'$ precisely when $f(x)\leq f(x')$, for any $x,x'\in \cX$. Denote by $\Emb_\cX^\cP$ the set of all order-embeddings of $\cX$ into $\cP$ up to automorphism of $\cX$. 
\end{defn}

\begin{rem}
To avoid redundancy, we chose to consider the set $\Emb_\cX^\cP$  of order-embeddings up to automorphism of $\cX$. For instance, if the posets $\cX$ and $\cP$ are given by 

\[
  \scalebox{0.8}{
  \begin{tikzpicture}
\begin{scope}[xshift=5cm]  \node(x) at (-1,0) {$\cP=$};
\node (A) at (0,0) {$a$};
  \node (B) at (1,0.5) {$b$};
  \node (C) at (1,-0.5) {$c$};
\node (D) at (2,0) {$d$};

\draw[->] (A)--(B);
\draw[->] (B)--(D);

\draw[->] (A)--(C);
\draw[->] (C)--(D);
\end{scope}

\node(x) at (0,0) {$\cX=$};
\node (A) at (1,0) {$1$};
\node (B) at (2,0.5) {$2$};
\node (C) at (2,-0.5) {$3$};
\draw[->] (A)--(B);
\draw[->] (A)--(C);
 \end{tikzpicture}}\]
 then $|\Emb_\cX^\cP| = 1$, even if there exist two different order-embeddings of $\cX$ into $\cP$.
\end{rem}

Any $f\in \Emb_\cX^\cP$ induces a restriction functor $f^*:\mod \cP\to \mod \cX$ which is exact. It is also dense since order-embeddings are injective. Therefore $f^*$ induces a surjective $\bZ$-linear map $\Ksp(\cP)\to \Ksp(\cX)$.

\begin{defn}\label{def:embedding}
Let $\cX$ be a poset of finite representation type,  $\cP$ be a finite poset and,  $\rE$ be a subset of $\Emb_\cX^\cP$. We denote by $\rI:=\Ind\cX$. 

We define $\bm_{\cX,\cP}^{\rE}:\Ksp(\cP)\to \bZ^{\rE\times \rI}$ by 
$$ \left(\bm_{\cX,\cP}^{\rE}(M)\right)_{(f,U)}:={\rm mult}_{\cX}(U,f^*M)\quad \textrm{for } (f,U)\in \rE\times \rI.$$  

We define $\bdh^E_{\cX,\cP}:\Ksp(\cP)\to \bZ^{E\times I}$ by 
$$\left(\bdh_{\cX,\cP}^E(M)\right)_{(f,U)}:=\dim \Hom_{\cX}(U,f^*M) \quad \textrm{for } (f,U)\in \rE\times \rI.$$
When $\rE=\Emb_{\cX}^\cP$ we write $\bm_{\cX,\cP}$ and $\bdh_{\cX,\cP}$ instead of $\bm_{\cX,\cP}^\rE$ and $\bdh_{\cX,\cP}^\rE$.
\end{defn}

\begin{rem}
    See Remark~\ref{rem:dimhom} for a discussion of how to realize the invariant $\bdh_{\cX,\cP}^\rE$ as a dim-Hom invariant (as defined in Example~\ref{example:invariants}(4)).
\end{rem}

We show in Theorem~\ref{thm::mult_dimhom_invariants} that the invariants $\bm_{\cX,\cP}^\rE$ and $\bdh_{\cX,\cP}^{\rE}$ are equivalent. The aim of this paper is to study these (equivalent) invariants. We first review a pair of examples which have appeared recently in the literature.

\begin{ex}\label{ex:restrict}
    \begin{enumerate}
        \item Consider the poset
        \[
  \scalebox{0.8}{
  \begin{tikzpicture}
\node(x) at (-1,0) {$\cX=$};
\node (A) at (0,0) {$a$};
  \node (B) at (1,0.5) {$b$};
  \node (C) at (1,-0.5) {$c$};
\node (D) at (2,0) {$d$};

\draw[->] (A)--(B);
\draw[->] (B)--(D);

\draw[->] (A)--(C);
\draw[->] (C)--(D);
\end{tikzpicture}}
    \]
        This poset has 10 indecomposable modules, all of which are interval modules. Of these, there are two which are not rectangle modules, namely $\mathbb{I}_{\{a,b,c\}}$ and $\mathbb{I}_{\{b,c,d\}}$.
        For $\cP$ a product of two totally ordered sets, it is shown in \cite{BLO} that a $\cP$-module $M$ is rectangle decomposable if and only if it satisfies
        $$\left(\bm_{\cX,\cP}(M)\right)_{(f,U)} = 0 \text{ for all $(f,U) \in \Emb_\cX^\cP \times \{\mathbb{I}_{\{a,b,c\}},\mathbb{I}_{\{b,c,d\}}\}$}.$$
        Note that the authors of \cite{BLO} do not assume $\cP$ to be finite.
        \item The ``generalized rank invariant'' of \cite{KM} and the ``compressed multiplicities'' of \cite{AENY2} both record the multiplicity of the sincere interval modules $\mathbb{I}_{\cX}$ over certain subsets $\cX \subseteq \cP$. (The authors of \cite{AENY2} assume that $\cP$ is a product of two finite totally ordered sets and the authors of \cite{KM} work over an arbitrary locally finite poset.) See \cite[Section~3.2]{BBH} for a summary of these and related results. In Section~\ref{sec:families}, we will similarly examine situations where we allow the poset $\cX$ to vary and record only the multiplicity of the sincere interval modules.
    \end{enumerate}
\end{ex}

\begin{thm}\label{thm::mult_dimhom_invariants}
Let $\cX$ be a representation-finite poset and $\cP$ be a finite poset. Then for any $\rE\subseteq \Emb_{\cX}^{\cP}$ the invariants $\bm^{\rE}_{\cX,\cP}$ and $\bdh^{\rE}_{\cX,\cP}$ are equivalent. 
\end{thm}

\begin{proof}
By Proposition \ref{prop::mult-dimhom}, there is an invertible map $\Phi:\bZ^{\rI}\to \bZ^{\rI}$ such that $\Phi\circ \bm_{\cX}=\bdh_{\cX}^{\rI}$.  We then obtain the following commutative diagram 
$$\xymatrix{\Ksp(\cP)\ar[rr]^{\bm_{\cX,\cP}^{\rE}}\ar@{=}[d] && \bZ^{\rE\times I}\ar[d]^{\Phi^{\rE}}\\ \Ksp(\cP)\ar[rr]^{\bdh_{\cX,\cP}^{\rE}} && \bZ^{\rE\times \rI} }$$
where $\Phi^{\rE}$ is 
 the isomorphism sending $x\in \bZ^{\rE\times \rI}$ to the map $\Phi^\rE(x)\in \bZ^{\rE\times \rI}$ which is defined by applying, for a fixed $f\in \rE$, the isomorphism $\Phi$ to the map $x_{(f,-)} \in \bZ^{\rI}$; in formulas:
$$\Phi^\rE(x)_{(f,M)}:=\Phi (x_{(f,-)})_M \quad \textrm{for } f\in \rE,\ M\in \rI,\ x\in \bZ^{\rE\times \rI} .$$
\end{proof}

\begin{ex}\label{Ex::mult2}
\begin{enumerate}
\item 
Let $\cX_1$ be the poset with one element. Then $\Emb_{\cX_1}^\cP$ is naturally in bijection with $\cP$, and $\rI=\Ind\cX$ has one element. For $a\in \cP$, one immediately checks that $\bm_{\cX_1,\cP}=\bdim_\cP=\bdh_{\cX_1,\cP}$.

\item Let $\cX_2$ be the poset $\{1\to 2\}$. It is well-known that the poset $\cX_2$ is of representation-finite type, with 
$\rI =\Ind \cX_2=\{P_2 = \bI_{\{2\}},
P_1 = \bI_{\{1,2\}},H_{1,2} = \bI_{\{1\}}\}$.

Let $f:\cX_2\to \cP$ be an order-embedding. Then one easily checks that for any $M\in \mod \cP$
$$\bm_{\cX_2,\cP}(M)_{(f,U)}=\left\{\begin{array}{ll} \dim M_{f2}- \rk M_{f1\to f2 } & \textrm{ if } U=\bI_{\{2\}} \\ 
\rk M_{f1\to f2 } & \textrm{ if } U= \bI_{\{1,2\}}\\ 
\dim M_{f1}- \rk M_{f1\to f2 } & \textrm{ if } U= \bI_{\{1\}}
\end{array}\right.\qquad (3.7.1)$$

And we have 

$$\bdh_{\cX_2,\cP}(M)_{(f,U)}=\left\{\begin{array}{ll} 
\dim M_{f2} & \textrm{ if } U=\bI_{\{2\}} \\ 
\dim M_{f1} & \textrm{ if } U=\bI_{\{1,2\}}\\ 
\dim M_{f1}-\rk M_{f1\to f2} & \textrm{ if } U= \bI_{\{1\}}
\end{array}\right.\qquad (3.7.2)$$
The isomorphism $\Phi$ which yields the equivalence between $\bm_{\cX_2,\cP}^f$ and $\bdh_{\cX_2,\cP}^f$
is triangular with diagonal entries $1$ as constructed in  Proposition \ref{prop::mult-dimhom}, with  $U_1= \bI_{\{2\}},
U_2 = \bI_{\{1,2\}},U_3 = \bI_{\{1\}}$\; :
$$\Phi = \begin{pmatrix} 1 & 0 & 0\\ 1 & 1 & 0\\ 0 & 1 & 1\end{pmatrix}. $$

\end{enumerate}
\end{ex}

From part (2) of this example we can also deduce that multiplicities can be calculated from ranks: Equation (3.7.1) shows that composing the vector of ranks with the matrix $$\Phi_f=\begin{pmatrix} 1 & -1 & 0\\ 0 & \;\;\; 1 & 0\\ 0 & -1 & 1\end{pmatrix}$$  yields the vector of multiplicities.

We formalize this observation in  the following corollary:

\begin{coro}\label{corX_2Rank}
Let $\cP$ be a finite connected poset with at least 2 elements. Then the invariants 
$\brk_\cP$ and $\bm_{\cX_2,\cP}$ are equivalent.
\end{coro}

\begin{proof}

 Let $\rE = \Emb_{\cX_2}^\cP$ and consider the map $p:\rE\times \rI\to \{\leq_\cP\}$ defined by 
$$ p(f,U) = \left\{ \begin{array}{ll}  (f2,f2) & \textrm{if }U=\bI_{\{2\}} \\
 (f1,f2) & \textrm{if } U=\bI_{\{1,2\}}\\ 
 (f1,f1) & \textrm{if } U = \bI_{\{1\}}
\end{array}\right. $$
Since the poset $\cP$ is connected and has at least two elements, this map is surjective and induces an injective morphism $ p^*: \;\bZ^{\{\leq_{\cP}\}}\to \bZ^{\rE\times \rI}$. Let $\Phi: \bZ^{\rE\times \rI}\to \bZ^{\rE\times \rI}$ defined by the matrix 
$\Phi_f$ above
for any $f\in \rE$ in the basis of $\bZ^{\rE\times \rI}$ induced by the basis $(\bI_{\{2\}}, \bI_{\{1,2\}},\bI_{\{1\}})$ of $\bZ^I$.

Then we have $\Phi\circ p^* \circ \brk_{\cP}= \bm_{\cX_2,\cP}$. Moreover since $\brk_{\cP}$ is surjective, and since $\Phi\circ p^*$ is injective, we obtain that $\Phi\circ p^*$ is an isomorphism from $\bZ^{\{\leq_\cP\}}$ to the image of $\bm_{\cX_2,\cP}$.
\end{proof}

\begin{rem}\label{rem::multX2_redundant}
Since $\cX_2$ has three indecomposables, the set $\rE \times I= \Emb_{\cX_2}^\cP\times \Ind \cX_2 $ has three times the cardinality  of  $\{<_\cP\}$, where $\{<_\cP\}:=\{(a,b)\in \cP^2| a<b\}$. This is in general strictly greater than $|\{\leq_\cP\}|=|\{<_\cP\}|+|\cP|$, so the invariant $\bm_{\cX_2,\cP}$ is not surjective. We will see in Section~\ref{sec:families} a way to reduce this redundant information.
\end{rem}

\begin{prop} Let $\cP$ be a connected poset  such that any two comparable elements lie in a chain of length $3$. Let $\cX_3$ be the poset $\{1\to 2\to 3\}$. Then the invariants $\brk_{\cP}$ and $\bm_{\cX_3,\cP}$ are equivalent.
\end{prop}

\begin{proof}
 There are $6$ indecomposable objects in $\mod \cX_3$ which we label by 
$$ \rI= \left\{  H_{1,2} = \bI_{\{1\}}, H_{2,3} = \bI_{\{2\}}, P_3 = \bI_{\{3\}}, H_{1,3} = \bI_{\{1,2\}}, P_2 = \bI_{\{2,3\}}, P_3 = \bI_{\{1,2,3\}}\right\}.$$

One computes easily that

\[\left(\bm_{\cX_3,\cP}(M)\right)_{(f,U)} = \left\{\begin{array}{ll} \dim M_{f1}-\rk M_{f1\to f2} & \textrm{ if }U=\bI_{\{1\}}\\
\dim M_{f2}-\rk M_{f1\to f2}-\rk M_{f2\to f3}+\rk M_{f1\to f3} & \textrm{ if }U=\bI_{\{2\}}\\
\dim M_{f3}-\rk M_{f2\to f3}& \textrm{ if }U=\bI_{\{3\}}\\
\rk M_{f1\to f2} -\rk M_{f1\to f3}& \textrm{ if }U=\bI_{\{1,2\}}\\
\rk M_{f2\to f3} -\rk M_{f1\to f3}& \textrm{ if }U=\bI_{\{2,3\}}\\
\rk M_{f1\to f3}& \textrm{ if }U=\bI_{\{1,2,3\}}
\end{array}\right.\]

For comparable elements $a<b$ in $\cP$, by hypothesis there exists $c\in \cP$ with $c<a$, with $b<c$ or with $a<c<b$. Therefore, there exists $f\in \Emb_{\cX_3}^\cP$  with $a$ and $b$ in its image. Using this embedding $f$, one can easily check the equivalence.  
\end{proof}

\begin{rem}\label{rem:relationsKernel} Note that here again, in general, the cardinality of $\Emb_{\cX_3}^\cP\times \rI$ is strictly greater than the cardinality of $\{\leq_{\cP}\}$. Therefore the invariant $\bm_{\cX_3,\cP}$ is not surjective. For example, if $\cP$ is as follows 
\[
  \scalebox{0.8}{
  \begin{tikzpicture}

\begin{scope}[xshift=9cm]
\node at (-1,0) {$\cP=$};
  \node (A) at (0,0) {$a$};
  \node (B) at (1,0) {$b$};
  \node (C) at (2,0.5) {$c$};
  
\node (D) at (2,-0.5) {$d$};

\draw[->] (A)--(B);
\draw[->] (B)--(C);
\draw[->] (B)--(D);
\end{scope}
 \end{tikzpicture}}\] 
one has $|\Emb_{\cX_3}^\cP|=2$, so $|\Emb_{\cX_3}^\cP\times \rI|=12$. However the rank of the invariant is $9=|\{\leq_{\cP}\}|$.  If we denote by $f_c$ ( resp.$f_d$) the element in $\rE$ with $f_c(3)=c$ (resp. $f_d(3)=d$), and $\bm = \bm_{\cX_3,\cP}$, then one has 
$$\bm (\bI_{\{b,c,d\}})_{(f,U)}= \left\{\begin{array}{ll} 1 & \textrm{ for  } U=\bI_{\{2,3\}}\\ 0 & \textrm{else}
\end{array}\right. $$

$$\bm (\bI_{\{b,c\}})_{(f,U)}= \left\{\begin{array}{ll} 1 & \textrm{ for  $(f,U)=(f_c,\bI_{\{2,3\}})$ or $(f_d, \bI_{\{2\}})$} \\ 0 & \textrm{else}
\end{array}\right. $$
 and 
$$\bm (\bI_{\{b\}})_{(f,U)}= \left\{\begin{array}{ll} 1 & \textrm{ for  $U=\bI_{\{2\}}$} \\ 0 & \textrm{else}
\end{array}\right. $$

Hence we observe the following relation 
$$\bm (\bI_{\{b,c,d\}})-\bm (\bI_{\{b,c\}})-\bm (\bI_{\{b,d\}})+\bm (\bI_{\{b\}})=0.$$

Similar computations give $$\bm (\bI_{\{a,b,c,d\}})-\bm (\bI_{\{a,b,c\}})-\bm (\bI_{\{a,b,d\}})+\bm (\bI_{\{a,b\}})=0;$$
$$\bm (M)-\bm (\bI_{\{a,b\}})-\bm (\bI_{\{b,c\}})-\bm (\bI_{\{b,d\}})+\bm (\bI_{\{b\}})=0,$$
where $M$ denotes the indecomposable $\cP$-module with dimension vector $(1,2,1,1)$. One further may check that the family in $K_0^{\rm sp}(\cP)$
$$([\bI_{\{b,c,d\}}\oplus\bI_{\{b\}}]-[\bI_{\{b,c\}}\oplus\bI_{\{c,d\}}], \ [\bI_{\{a,b,c,d\}}\oplus\bI_{\{a,b\}}]-[\bI_{\{a,b,c\}}\oplus\bI_{\{a,b,d\}}],\  [M\oplus\bI_{\{b\}}]-[\bI_{\{a,b\}}\oplus\bI_{\{b,c\}}\oplus\bI_{\{b,d\}}])$$ forms a basis of the kernel of $\bm: K_0^{\rm sp}(\cP)\to \bZ^{\rE\times \rI}$.
\end{rem}

\begin{prop}\label{prop:X_3Rank}
 Let $\cX'_3$ and $\cX''_3$ be the following posets 
\[
  \scalebox{0.8}{
  \begin{tikzpicture}
  
\begin{scope}[xshift=4cm]

\node(x) at (0,0) {$\cX'_3=$};
\node (A) at (1,0) {$1$};
\node (B) at (2,0.5) {$2$};
\node (C) at (2,-0.5) {$3$};
\draw[->] (A)--(B);
\draw[->] (A)--(C);
\end{scope}

\begin{scope}[xshift=8cm]
\node(x) at (0,0) {$\cX''_3=$};
\node (A) at (1,0.5) {$1$};
\node (B) at (1,-0.5) {$2$};
\node (C) at (2,0) {$3$};
\draw[->] (A)--(C);
\draw[->] (B)--(C);
\end{scope}
 \end{tikzpicture}}\]
Let $\cP$ be a connected poset such that any pair $a<b\in \cP$ is in the image of an order-embedding $\cX'_3\to \cP$ (resp. of an order embedding $\cX''_3\to \cP$), then the invariant
$\bm_{\cX'_3,\cP}$ (resp. $\bm_{\cX''_3,\cP}$) is strictly finer than the rank invariant $\brk_\cP$. 

\end{prop}

\begin{proof}

The indecomposable $\cX'_3$-modules are 
$$\Ind\cX'_3=\left\{S_1 = \mathbb I_{\{1\}},P_2 = \mathbb I_{\{2\}},P_3 = \mathbb I_{\{3\}},H_{1,3} = \mathbb I_{\{1,2\}}, H_{1,2} = \mathbb I_{\{1,3\}},P_1 = \mathbb I_{\{1,2,3\}}\right\}.$$

Now for $f \in \Emb_{\cX'_3}^\cP$, denote by $\rk M_{f1\to f2\oplus f3}$ the rank of the map %
$$\xymatrix{M_{f1}\ar[rr]^-{\begin{pmatrix} M_{f1\to f2}\\ M_{f1\to f3}\end{pmatrix}} && M_{f2}\oplus M_{f_3}}.$$

Then one computes that
\[\left(\bm_{\cX'_3,\cP}(M)\right)_{(f,U)} = \left\{\begin{array}{ll} \dim M_{f1}-\rk M_{f1\to f2\oplus f3}& \textrm{ if }U=\mathbb I_{\{1\}}\\
\dim M_{f2}-\rk M_{f1 \to f2} & \textrm{ if }U=\mathbb I_{\{2\}}\\
\dim M_{f3}-\rk M_{f1\to f3} & \textrm{ if }U=\mathbb I_{\{3\}}\\
\rk M_{f1\to f2\oplus f3}-\rk M_{f1\to f3}& \textrm{ if }U=\mathbb I_{\{1,2\}}\\
\rk M_{f1\to f2\oplus f3}-\rk M_{f1\to f2} & \textrm{ if }U=\mathbb I_{\{1,3\}}\\\rk M_{f1\to f2} +\rk M_{f1\to f3} -
\rk M_{f1\to f2\oplus f3}& \textrm{ if }U=\mathbb I_{\{1,2,3\}}
\end{array}\right.\]
\bigskip

Let $a<b$ be in $\cP$. Then by hypothesis, there exists an embedding $f:\cX'_3\to \cP$ with $f(1)=a$ and $f(2)=b$. This yields

$$\left(\brk_\cP(M)\right)_{(a,b)}= \left(\bm_{\cX'_3,\cP}(M)\right)_{\left(f, \bI_{\{1,2\}}\right)}+\left(\bm_{\cX'_3,\cP}(M)\right)_{\left(f,\bI_{\{1,2,3\}}\right)}.$$
For $a\in \cP$, by 
 connectedness, there exists $f\in \Emb_{\cX'_3}^\cP$ with either $f(1)=a$ or with $f(2)=a$. 

If $f(1)=a$, we check that 
 $\left(\brk_\cP(M)\right)_{(a,a)}$ is in the free abelian group generated by $$\left\{\delta_{(f, \bI_{\{1\}})}, \delta_{(f, \bI_{\{1,2\}})}, \delta_{(f, \bI_{\{1,3\}})}, \delta_{(f, \bI_{\{1,2,3\}})}\right\},$$ where $\left\{\delta_{(f,U)} \mid (f,U)\in \rE\times \rI\right\}$ denotes the canonical dual basis in $\bZ^{\rE\times \rI}$. 

If $f(2)=a$, one can also check that  $\left(\brk_\cP(M)\right)_{(a,a)}$ is in the free abelian group generated by $$\left\{\delta_{(f, \bI_{\{2\}})}, \delta_{(f, \bI_{\{1,2\}})},\delta_{(f, \bI_{\{1,2,3\}})}\right\}.$$ 
Therefore we have $\brk_\cP\leq \bm_{\cX'_3,\cP}$. 

 To prove that the invariant $\bm_{\cX'_3,\cP}$ is strictly finer than $\brk_{\cP}$ consider the case 
$\cP=\cX'_3$.  The rank of $\brk_{\cP}$ is the cardinal of $\{\leq_{\cP}\}$ which is $5$, while the rank of $\bm_{\cX'_3}$ is $6$ since there are $6$ indecomposable modules in $\mod \cX'_3$. 

The proof is completely similar for $\cX''_3$. 
\end{proof}

\begin{rem}
\begin{enumerate}
\item

Note that here again, the cardinality of $\Emb_{\cX'_3}^\cP\times \rI$ is strictly greater than the rank of the invariant $\bm:=\bm_{\cX'_3,\cP}$. For instance if $\cP$ is the poset 

 \[
  \scalebox{0.8}{
  \begin{tikzpicture}

\begin{scope}[xshift=9cm]
\node at (-1,0) {$\cP=$};
  \node (A) at (0,0) {$a$};
  \node (B) at (1,0) {$b$};
  \node (C) at (2,0.5) {$c$};
  
\node (D) at (2,-0.5) {$d$};

\draw[->] (A)--(B);
\draw[->] (B)--(C);
\draw[->] (B)--(D);
\end{scope}
 \end{tikzpicture}}\] 
then $|\Emb_{\cX'_3}^\cP|= 2$ so $|\Emb_{\cX'_3}^\cP\times \rI|=12$, but one can check that the rank of $\bm$ is $10$. 
 Denoting $f_a$ (resp. $f_b$) the embedding $\cX\to \cP$ with $f1=a$ (resp. $f1=b$), then one computes 
\[\bm (M)_{(f,U)}= \left\{\begin{array}{ll} 1 & \textrm{if $(f,U)=(f_a,\bI_{\{1,2,3\}})$, $(f_b,\bI_{\{1,2\}})$ or $(f_b,\bI_{\{1,3\}})$}\\ 0 & \textrm{else}\end{array}\right.\]
where $M$ is the indecomposable $\cP$-module with dimension vector $(1,2,1,1)$.

We therefore have the relations 
$$\bm(\bI_{\{a\}})+\bm (\bI_{\{ b\}})=\bm (\bI_{\{a,b\}}),$$

$$\bm(M)-\bm (\bI_{\{a,b,c,d\}})+\bm (\bI_{\{b,c,d\}})-\bm(\bI_{\{b,c\}})-\bm(\bI_{\{b,d\}})=0,$$
that define a basis of the kernel of $\bm:K_0^{\rm sp}(\cP)\to \bZ^{\rE\times \rI}$ as in Remark \ref{rem:relationsKernel}.

\item The hypothesis on $\cP$ in Proposition~\ref{prop:X_3Rank} implies the weaker property that $$\bigcup_{f \in \Emb_{\cX'_3}^\cP} \Im f = \cP.$$
We note that this weaker property is not sufficient in order to deduce the fact that $\bm_{\cX_3',\cP}$ is strictly finer than $\brk_\cP$. 
Indeed consider the case where $\cP$ is given as above. 
This poset satisfies the weaker property above, but there is no embedding $f:\cX'_3\to \cP$ with $f(1)=a$ and $f(2)=b$.
As mentioned above, the family $\left\{\bm(\bI_{\{a\}}), \bm(\bI_{\{b\}}), \bm (\bI_{\{a,b\}})\right\}$ is not free, but one easily checks that the family $\left\{\brk_\cP(\bI_{\{a\}}), \brk_\cP(\bI_{\{b\}}), \brk_\cP (\bI_{\{a,b\}})\right\}$ is free. Thus the invariant $\bm_{\cX_3',\cP}$ is not finer than $\brk_\cP$, and in fact the two invariants are incomparable
in this case. Similar examples occur when $\cX_3'$ is replaced with $\cX_3''$.
\end{enumerate}
\end{rem}

The condition on $\cP$ in Proposition \ref{prop:X_3Rank} may seem strong.  
For instance, it is not satisfied when $\cP$ is an $n\times m$ grid, $\cP  = \{1,2, \ldots , n\} \times \{1,2, \ldots , m\}$, since the ordered pairs between vertices of the form $(i,m)$ are not in the image of an order-embedding of $\cX_3'$. 
\begin{coro}

Let $\cP  = \{1,2, \ldots , n\} \times \{1,2, \ldots , m\}$ be an $n\times m$ grid with $n,m >1.$ 
Then there exists an embedding $\cP\to \cP'$ and an invariant
$\overline{\bm}_{\cX'_3,\cP',\cP}$ (resp. $\overline{\bm}_{\cX''_3,\cP',\cP}$) which is strictly finer than the rank invariant $\brk_\cP$. 
 
\end{coro}

\begin{proof}
We consider the case of $\cX'_3$.
Embed the poset $\cP$ into a poset $\cP'$ by adding a staircase to the top and to the left, as illustrated below by the example $\cP  = \{1,2,3\} \times \{1,2,3\}$
   
     \[
  \scalebox{0.8}{
  \begin{tikzpicture}

\begin{scope}[xshift=9cm]
\node at (-1,3) {$\cP'=$};
  \node (a1) at (1,1) {$(1,1)$};
  \node (a2) at (3,1) {$(2,1)$};
  \node (a3) at (5,1) {$(3,1)$};
    \node (a4) at (7,1) {$(4,1)$};
      \node (a5) at (9,1) {$(5,1)$};
 \node (b1) at (1,2) {$(1,2)$};
  \node (b2) at (3,2) {$(2,2)$};
  \node (b3) at (5,2) {$(3,2)$};
    \node (b4) at (7,2) {$(4,2)$};
   \node (c1) at (1,3) {$(1,3)$};
  \node (c2) at (3,3) {$(2,3)$};
  \node (c3) at (5,3) {$(3,3)$};
     \node (d1) at (1,4) {$(1,4)$};
  \node (d2) at (3,4) {$(2,4)$};
    \node (e1) at (1,5) {$(1,5)$};
\draw[->] (a1)--(a2);
\draw[->] (a2)--(a3);
\draw[->] (a3)--(a4);
\draw[->] (a4)--(a5);
\draw[->] (b1)--(b2);
\draw[->] (b2)--(b3);
\draw[->] (b3)--(b4);
\draw[->] (c1)--(c2);
\draw[->] (c2)--(c3);
\draw[->] (d1)--(d2);

\draw[->] (a1)--(b1);
\draw[->] (a2)--(b2);
\draw[->] (a3)--(b3);
\draw[->] (a4)--(b4);
\draw[->] (b1)--(c1);
\draw[->] (b2)--(c2);
\draw[->] (b3)--(c3);
\draw[->] (c1)--(d1);
\draw[->] (c2)--(d2);
\draw[->] (d1)--(e1);
\end{scope}
 \end{tikzpicture}}\]

Clearly every pair $a < b \in \cP'$ is in the image of an order-embedding $\cX'_3 \to \cP'$, and therefore by Proposition \ref{prop:X_3Rank}, 
the invariant
$\bm_{\cX'_3,\cP'}$ is strictly finer than the rank invariant $\brk_{\cP'}$. 
The fact that the invariants $\bm_{\cX'_3,\cP'}$ and $\brk_{\cP'}$ are not equivalent is shown in the proof of Proposition \ref{prop:X_3Rank} by embedding one single copy of $\cX'_3 \to \cP'$, and this can be realized  within the boundaries of $\cP$ as soon as $n,m >1.$
\end{proof}

The examples $\cX_1, \cX_2,\cX_3,\cX'_3,\cX''_3$ all have the property that they admit a unique  indecomposable representation which is non-zero at all vertices (such an indecomosable is called {\em sincere}).
We present now an example where this property fails, but we show that still the multiplicity invariant can be computed purely in terms of linear algebra of the representation $M$:

\begin{ex}\label{Ex::mult4}
Consider the poset $\cX_4$ of type $D_4$

 \[
  \scalebox{0.8}{
  \begin{tikzpicture}

\begin{scope}[xshift=9cm]
\node at (-1,0) {$\cX_4=$};
  \node (A) at (0,0) {$a$};
  \node (B) at (1,0) {$b$};
  \node (C) at (2,0.5) {$c$};
  
\node (D) at (2,-0.5) {$d$};

\draw[->] (A)--(B);
\draw[->] (C)--(B);
\draw[->] (D)--(B);
\end{scope}
 \end{tikzpicture}}\]     
 Then $\cX_4$ has two sincere indecomposable representations, one being the characteristic representation $\bI_{\cX_4}$, the other one is given by three one-dimensional subspaces in general position, placed at vertices $a,c,d$, of a two-dimensional vector space at vertex $b$. 
 Still, their multiplicities can be described in terms of the linear maps of the $\cP$-module $M$ restricted to $\cX_4$, for example the multiplicity of $\bI_{\cX_4}$ is given by 
 \[ \bm_{\cX_4,\cP} (\bI_{\cX_4}) = \dim (\Im M_{fa \to fb} \cap \Im M_{fc \to fb} \cap \Im M_{fd \to fb}).\]
\end{ex}

\section{Invariants and exact structures}

\subsection{Exact structures and relative projectives}

Exact structures have been introduced in
\cite{quillen} as an axiomatic framework allowing to use methods from homological algebra relative to a fixed class of short exact sequences. Besides this axiomatic approach, there are a number of ways to describe an exact structure  $\cE$, such as  by certain subfunctors $\cE$ of the bifunctor $\Ext_\cP^1$, or by specifying what are the projective objects relative to  $\cE$, or by the Auslander-Reiten sequences belonging to $\cE$.

Let $\cE$ be a class of short exact sequences in $\mods \cP$. We assume that $\cE$ is closed under isomorphisms. If a short exact sequence $\eta = (0 \to A \xrightarrow{f} E \xrightarrow{g} B \to 0)$
belongs to $\cE$, we say that $\eta$ is an ($\cE$-)\emph{admissible short exact sequence}, and also that $f$ is an ($\cE$-)\emph{admissible monomorphism} and $g$ is an ($\cE$-)\emph{admissible epimorphism}. The closure under isomorphisms then implies that the classes of admissible short exact sequences, admissible monomorphisms, and admissible epimorphisms uniquely determine one another.
Quillen's definition can then be rephrased as follows:
\begin{defn}\label{def:exact}
A class $\cE$  of short exact sequences in $\mods \cP$ is said to be an \emph{exact structure} on $\mods\cP$ if all of the following hold:
 \begin{itemize}
	\item[(E0)]  $\cE$ contains all split exact sequences.
        \item[(E1)] $\cE$ is closed under compositions of admissible monomorphisms, that is, if $f: X \to Y$ and $g: Y \to Z$ are admissible monomorphisms, then  $g \circ f : X \to Z$ is also an admissibe monomorphism.
        Likewise, $\cE$ is closed under compositions of admissible epimorphisms.
        \item[(E2)] $\cE$ is closed under pushouts: If $f: X \to Y$ is an admissible monomorphism,
        and $X\xrightarrow{h} W$ is any morphism in $\cA$, then the pushout of $f$ along $h$ yields a short exact sequence in $\cE$.
        Likewise, $\cE$ is closed under pullbacks.
	\end{itemize}
\end{defn}

 The two extreme examples of exact structures are the \emph{split exact structure} $\cE^\mathrm{sp}$, which contains only the split short exact sequences and corresponds to the subfunctor 0 of $\Ext^1_\cP$, and the exact structure $\cE^{\mathrm{all}}$, which contains all short exact sequences and corresponds to $\Ext^1_\cP$ as a subfunctor of itself.

\begin{prop}
    Let $\cP$ and $\cX$ be finite posets.  
    Let $\cE$ be an exact structure on $\mods \cX$ and let $\rE\subseteq \Emb_{\cX}^{\cP}$. Let $\Psi_\rE(\cE)$ denote the class of short exact sequences
    $$0 \rightarrow L \rightarrow M \rightarrow N \rightarrow 0$$
    in $\mods \cP$ which have the property that
    $$0 \rightarrow f^*L \rightarrow f^* M \rightarrow f^* N \rightarrow 0$$
    is $\cE$-exact for all $f \in \rE$. Then $\Psi_\rE(\cE)$ is an exact structure on $\mods \cP$.
\end{prop}

\begin{proof}
We verify the three points from definition \ref{def:exact} for all $f \in \rE$, using that the restriction functor $f^* : \mods \cP \to \mods \cX$ is additive and exact.

\begin{itemize}
     \item[(E0)] A split exact sequence is sent by $f^*$ to a split exact sequence, which belongs necessarily to $\cE$, therefore $\Psi(\cE)$ contains all split exact sequences.
     \item[(E1)] This is clear from the definition of $\Psi(\cE)$ and the fact that (E1) holds for $\cE$.
     \item[(E2)] Follows from the fact that (E2) holds for $\cE$ and that $f^*$ preserves pushouts since it is an exact functor. \qedhere
\end{itemize}
\end{proof}

\begin{defn}
Let $\cE$ be an exact structure on $\mod \cP$. We denote by ${\rm K}_0(\cE)$ the quotient of $\Ksp(\cP)$ by the relations 
$$[L]+[N]-[M]\textrm{ for any short exact sequence } (0 \rightarrow L \rightarrow M \rightarrow N \rightarrow 0) \textrm{ in }\cE.$$
\end{defn}

\begin{prop} Let $\cX$ be a representation-finite poset and $\cP$ be a finite poset. Recall that $\rI=\Ind \cX$ and fix $\rE\subseteq \Emb_{\cX}^{\cP}$. Then denote by $\cE^{\rE}_{\cP}:=\Psi_\rE(\cE^{\rm sp}_\cX)$, where $\cE_\cX^{\rm sp}$ is the split exact structure on $\mod \cX$. Then we have a factorization  of the commutative square given in Theorem \ref{thm::mult_dimhom_invariants}

$$\xymatrix{\Ksp(\cP)\ar@{=}[dd] \ar[rr]^{\bm_{\cX,\cP}^{\rE}}\ar@{->>}[dr] & & \bZ^{\rE\times \rI}\ar[dd]^{\sim} \\ & {\rm K}_0(\cE_{\cP}^{\rE})\ar[ur]\ar[dr] & \\ \Ksp(\cP) \ar[rr]^{\bdh_{\cX,\cP}^{\rE}}\ar@{->>}[ur] & & \bZ^{\rE\times \rI}}$$
\end{prop}

\begin{proof}Let $f:\cX\to \cP$ be in $\rE$ and  $(0 \rightarrow L \rightarrow M \rightarrow N \rightarrow 0)$ be an $\cE$-exact sequence. Then by definition we have an isomorphism $f^*M\simeq f^*N\oplus f^*L$. Therefore for each  $U\in \Ind \cX$ we have $$\bm_{\cX,\cP}^\rE(M)_{(f,U)}:=\bm_\cX(f^*M)_U=\bm_{\cX}(f^*N)_U+\bm_\cX(f^*L)_U,$$
which proves the desired factorization.
\end{proof}

\subsection{Relative projectives}

    The projective objects (relative to $\cE$) are then described as follows:
\begin{defn}\label{def:projectives}
Let $\cE$ be an exact structure on $\mods\cP$.
 \begin{enumerate}
     \item We say an object $P \in \mods\cP$ is \emph{$\cE$-projective}, in symbols $P \in \proj(\cE)$, if $g_* = \Hom_\cP(P,g)$ is an epimorphism for all $Y, Z \in \mods\cP$ and for all admissible epimorphisms $g: Y \to Z$. Equivalently, $P \in \proj(\cE)$ if and only if $\Hom_\cA(P,-)$ sends admissible exact sequences to exact sequences.
     \item We say an object $I \in \mods\cP$ is \emph{$\cE$-injective}, in symbols $I \in \inj(\cE)$, if $f_* = \Hom_\cP(f,I)$ is an epimorphism for all $X, Y \in \mods\cP$ and for all admissible monomorphisms $f: X \to Y$. Equivalently, $I \in \inj(\cE)$ if and only if $\Hom_\cP(-,I)$ sends admissible exact sequences to exact sequences.
 \end{enumerate}
\end{defn}

The following is a classical result of \cite{AS,DRSS}. See \cite[Section~4]{BBH} for a summary tailored towards persistence theory.

\begin{prop}\label{prop:proj-exact} 
Let $\cU\subseteq \Ind \cP$. Denote by 
$$\cE_{\cU}:=\left\{ (0\to M\to N\overset{g}{\to} L\to 0) \mid  \forall U\in \cU, \ \Hom_\cP(U,g) \textrm{ is surjective.}\right\}$$
Then $\cE_\cU$ is an exact structure on $\mod \cP$, and we have 
$$\proj(\cE_\cU)=\cU \cup \proj (\cP).$$
\end{prop}

The following is classical, see \cite[Theorem~I.6.8]{BlueBook}.

\begin{prop}
Let $f:\cX\to \cP$ be an order embedding. Then the restriction functor $f^*:\mod \cP\to \mod \cX$ is exact and has a left adjoint $f_!:\mod \cX\to \mod \cP$ and a right adjoint $f_*:\mod \cX\to \mod \cP$. Moreover we have $f^*f_!={\rm Id}_{\mod \cX}  = f^*f_*$.
\end{prop}

The functors $f_*$  and $f_!$ are sometimes referred to as the \emph{induction} and \emph{coinduction} functors, or {\em right} and {\em left Kan extensions}, respectively. The induction functor preserves projective presentations and the coinduction functor preserves injective copresentations.

\begin{prop}\label{prop::relative-proj}
Let $\cX$ be a representation-finite poset and $\cP$ be a finite poset. Recall that $\rI =\Ind\cX$ and fix $\rE$ a subset of $\Emb_{\cX}^{\cP}$. Assume that $\cP=\bigcup_{f\in \rE}\Im f$. Define 
$$\cH^{\rE}:=\{f_!U\mid (f,U)\in \rE\times \rI\}.$$  Then the set $\cH^{\rE}$ is a subset of $\Ind\cP$ that contains the indecomposable projective modules. And we have 
$$\cE_{\cP}^{\rE}=\cE_{\cH^{\rE}} \textrm{ and } \proj(\cE_{\cP}^\rE)= \cH^\rE.$$
\end{prop}

\begin{proof}
First note that since $\cP=\bigcup_{f\in \rE}\Im f$, for any $p\in \cP$ there exists $x\in \cX$   and $f\in \rE$ such that $f(x)=p$. Since the functor $f_!$ sends projectives to projectives, we have that $f_!(P_x)=P_p$, that is the projective $P_p$ is in $\cH^\rE$. Moreover by the adjunction formula we have for any $U\in \Ind\cX$ 
$$\End_{\cP}(f_!U)\simeq \Hom_{\cX}(U,f^*f_!U)=\End_{\cX}(U)$$ which is local since $U\in \Ind \cX$. 

Let $$\xymatrix{\epsilon= (0\ar[r] & L\ar[r] & M\ar[r] & N\ar[r] & 0)}$$ be a short exact sequence of $\mod \cP$. Then by definition $\epsilon$ belongs to $\cE_\cP^E$ if and only if for any $f\in \rE$ the short exact sequence in $\mod \cX$ 
$$\xymatrix{ 0\ar[r] & f^*L\ar[r] & f^*M\ar[r] & f^*N\ar[r] & 0}$$ splits. This is equivalent to the fact that for any $U\in \Ind\cX$ this exact sequence yields a short exact sequence in $\mod k$
$$\xymatrix{ 0\ar[r] & \Hom_\cX(U,f^*L)\ar[r] & \Hom_{\cX}(U,f^*M)\ar[r] & \Hom_{\cX}(U,f^*N)\ar[r] & 0.}$$
By adjunction it is equivalent to the fact that for any $(f,U)\in \rE\times \rI$ the short exact sequence $\epsilon$ induces an exact sequence
$$\xymatrix{ 0\ar[r] & \Hom_\cP(f_!U,L)\ar[r] & \Hom_{\cP}(f_!U,M)\ar[r] & \Hom_{\cP}(f_!U,N)\ar[r] & 0,}$$ which is equivalent to the fact that $\epsilon\in \cE_{\cH^\rE}$. 
The fact that $\proj(\cE_{\cH^\rE})= \cH^\rE$ is now a direct consequence of Proposition \ref{prop:proj-exact}.
\end{proof}

\begin{rem}\label{rem:dimhom}
    Consider the setup of Proposition~\ref{prop::relative-proj}. It is an immediate consequence of the adjunction formula that the invariants $\bdh_{\cX,\cP}^{\rE}$ and $\bdh_{\cP}^{\cH^{\rE}}$ are equivalent. In particular, $\bdh_{\cX,\cP}^{\rE}$ (and thus also $\bm_{\cX,\cP}^{\rE}$ by Theorem~\ref{thm::mult_dimhom_invariants}) is a dim-Hom invariant in the sense of Example~\ref{example:invariants}(4). In order to prove the stronger property that these invariants are \emph{homological} (in the sense of \cite[Defintion~4.12]{BBH}), 
    one would need an argument showing that the exact structure $\cE_{\cP}^{\rE} = \cE_{\cH^\rE}$ has finite global dimension, which is difficult to establish in general.
    
\end{rem}

From Proposition~\ref{prop::relative-proj} we can deduce the following commutative diagram.
$$\xymatrix{\Ksp(\cP)\ar@{=}[dd] \ar[rrrr]^{\bm_{\cX,\cP}^{\rE}}\ar@{->>}[dr] && && \bZ^{\rE\times \rI}\ar[dd]^{\Phi} \\ & {\rm K}_0(\cE_{\cP}^{\rE})\ar[r] & \bZ^{\cH^{\rE}} \ar[r]^-{\sim}&\bZ^{\rE\times \rI/\sim_!}\ar@{^(->}[dr] & \\ \Ksp(\cP) \ar[rrrr]^{\bdh_{\cX,\cP}^{\rE}}\ar@{->>}[ur] \ar[urr]_{\bdh_\cP^{\cH^\rE}}& & &&\bZ^{\rE\times \rI}}$$
where the equivalence relation $\sim_!$ on $\rE\times \rI$ is defined by 
$$(f,U)\sim_!(g,V)\quad \Leftrightarrow\quad f_!U\simeq g_!V $$ and the map $\bZ^{\cH^\rE}\to \bZ^{\rE\times \rI/\sim_!}$ is induced by the map $\rE\times \rI\to \cH^\rE$ defined by $(f,U)\mapsto f_!U$.

\begin{ex}\label{ex::f!U modules}
\begin{enumerate}

\item 
Consider the poset $\cX=\cX_2 = \{1 \to 2 \}$ as in Example \ref{Ex::mult2}(2). Then the indecomposable $\cX_2$-modules are $P_1 = \bI_{\{1,2\}}$, $P_2 = \bI_{\{2\}}$ and $H_{1,2} = \bI_{\{1\}}$. The first two modules are projective, and the third is a hook
(and is the cokernel of the natural map $P_2 \rightarrow P_1$). Since $f_!$ is right exact and sends projectives to projectives, one obtains

$$\cH^{\rE} = \bigcup_{f\in \Emb_{\cX_2}^\cP} \{f_{!}U \ |\ U\in \mod \cX_2\}=\{P_{f1}, P_{f2}, H_{f1,f2}  \mid  f\in \Emb_{\cX_2}^\cP \}.$$ 
Thus Proposition \ref{prop::relative-proj} applied to the case $\cX=\cX_2$ recovers  the result  \cite[Theorem 4.11]{BOO} which desrcibes the indecomposable projective objects for the rank exact structure as the (lower) hook modules.

\item Consider $\cX=\cX'_3$ as in Proposition \ref{prop:X_3Rank}. 
 Using the fact that $f_!$ is right exact, and computing projective presentations of all indecomposables for $\cX=\cX'_3$, we obtain that
for $f\in \Emb_{\cX_2}^\cP$ we have 
$$\{f_!U\mid U\in I\}= \left\{P_{f1}, P_{f2}, P_{f3}, H_{f1,f2}, H_{f1,f3}, M:= \Coker (P_{f2}\oplus P_{f3}\to P_{f1})\right\}.$$
 Note that each of these is an interval module.
The modules $H_{f1,f2}$ and $H_{f1,f3}$ are hooks, while the module $M =\Coker (P_{f2}\oplus P_{f3}\to P_{f1})$ 
satisfies
$$\dim M_x = \left\{ \begin{array}{ll} k & \textrm{if } f(1)\leq x \textrm{ and } (x\ngeq f(2) \textrm{ or } x\ngeq f(3))\\ 0 & \textrm{else}.\end{array}\right.$$

 In the case where $\cP$ is a grid poset, the $\cP$-module $M$ looks as follows.
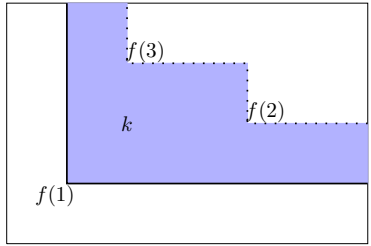
 \[
  \scalebox{0.8}{
  \begin{tikzpicture}

\draw (0,0)--(6,0)--(6,4)--(0,4)--(0,0);

\draw[blue!30, fill=blue!30] (1,1)--(6,1)--(6,2)--(4,2)--(4,3)--(2,3)--(2,4)--(1,4)--(1,1);

\draw[thick] (1,4)--(1,1)--(6,1);
\draw[thick, loosely dotted] (6,2)--(4,2)--(4,3)--(2,3)--(2,4);
\node at (0.8,0.8) {$f(1)$};
\node at (4.3,2.2) {$f(2)$};
\node at (2.3,3.2) {$f(3)$};
\node at (2,2) {$k$};

\end{tikzpicture}}\]

 \item The situation is surprisingly different for $\cX=\cX''_3$ as in Proposition \ref{prop:X_3Rank}. 
 
 Here we have an $\cX_3''$-module $\bI_{\cX_3''} = \Coker(P_3 \rightarrow P_1\oplus P_2)$. The corresponding $\cP$-modules $M = f_!\bI_{\cX_3''} = \Coker(P_{f3}\rightarrow P_{f1}\oplus P_{f2})$ satisfy
 $$\dim M_x = \left\{ \begin{array}{ll} k & \textrm{if } x\geq f(1)
 \\ k & \textrm{if } x\geq f(2) \textrm{ and } x\ngeq f(3)
\\ k & \textrm{if } x\geq f(3) \textrm{ and } x\ngeq f(2)
\\ k^2 & \textrm{if } x\geq f(2),\ x\geq f(3) \textrm{ and } x\ngeq f(1)
 \\ 0 & \textrm{else},\end{array}\right.$$
and hence are not  interval modules. 

 In the case where $\cP$ is a grid poset, the $\cP$-module $M$ looks as follows.
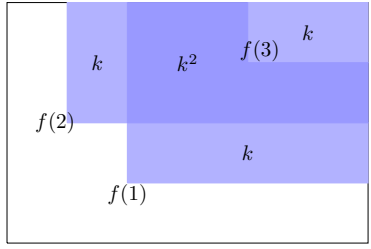
 \[
  \scalebox{0.8}{
  \begin{tikzpicture}

\draw (0,0)--(6,0)--(6,4)--(0,4)--(0,0);

\draw[blue!30, fill=blue!30] (1,2)--(2,2)--(2,4)--(1,4)--(1,2);

\draw[blue!30, fill=blue!30] (2,1)--(6,1)--(6,2)--(2,2)--(2,1);

\draw[blue!30, fill=blue!30] (4,3)--(6,3)--(6,4)--(4,4)--(4,3);

\draw[blue!40, fill=blue!40] (2,2)--(6,2)--(6,3)--(4,3)--(4,4)--(2,4)--(2,2);

\node at (4.2,3.2) {$f(3)$};
\node at (0.8,2) {$f(2)$};
\node at (2,0.8) {$f(1)$};
\node at (1.5,3) {$k$};
\node at (4,1.5) {$k$};
\node at (5,3.5) {$k$};
\node at (3,3) {$k^2$};

\end{tikzpicture}}\]

\end{enumerate}

\end{ex}

\subsection{Bases for invariants and relative projectives}

\begin{defn}\label{def::basis} Let $\Psi:\Ksp(\cP)\to \bZ^N$ be a finite invariant. A \emph{basis} of $\Psi$ is a set $\cB$ of indecomposable $\cP$-modules such that $\Psi$ induces an isomorphism of abelian group
$$\langle [U], U\in \cB\rangle \longrightarrow \Im\Psi,$$
where $\langle [U], U\in \cB\rangle$ is the free abelian subgroup of $\Ksp(\cP)$ generated by the $[U]$ for $U\in \cB$.
\end{defn}

As a consequence, if $\cB$ is a $\Psi$-basis, then for any module $M\in \mod \cP$ there exist unique integers $(\alpha_U)_{U\in\cB}$  such that 
$$\Psi(M)=\Psi \left(\bigoplus_{U,\alpha_U> 0} U^{\alpha_U}\right)-\Psi \left(\bigoplus_{U,\alpha_U <0} U^{-\alpha_U}\right).$$
Following the terminology of \cite{BOO}, we will refer to the pair $\left(\bigoplus_{U,\alpha_U> 0} U^{\alpha_U},\bigoplus_{U,\alpha_U <0} U^{-\alpha_U}\right)$ as the \emph{minimal $\Psi$-decomposition of $M$ with respect to $\cB$}. 

\begin{ex}\label{ex:bases}
\begin{enumerate}
\item The simple $\cP$-modules form a basis of $\bdim_\cP$. Moreover, since $\mod \cP$ has finite global dimension, the indecomposable projective $\cP$-modules also form an integer basis of $\bdim_\cP$.
\item Let $\cH \subseteq \Ind \cP$ denote the union of the set of hook modules and the set of indecomposable projectives (see Example~\ref{ex:intervals}). Then it is shown in \cite{BBH,BOO} that $\cH$ is a $\brk$-basis.         
\item It is shown in \cite{BOO} that the rectangle modules form another $\brk$-basis.
\end{enumerate}
\end{ex}

For use in Propositions~\ref{prop:basis_1} and~\ref{prop:basis_2} below, we recall the following well-known fact. This can also be taken as the definition of the Gabriel quiver containing no oriented cycles.

\begin{lem}
    Let $\cP$ be a finite poset and let $\cU$ a finite set of pairwise non-isomorphic indecomposable $\cP$-modules. Then the following are equivalent.
    \begin{enumerate}
        \item If there are nonzero morphisms
        $$X_1 \xrightarrow{f_1} X_2 \xrightarrow{f_2} \cdots \xrightarrow{f_{n-1}} X_n \xrightarrow{f_n} X_1$$
        with each $X_i \in \cU$, then $f_j$ is an isomorphism for all $j = 1,\ldots,n$.
        \item The Gabriel quiver of the algebra $\End_\cP(\bigoplus_{U \in \cU}U)$ contains no oriented cycles.
    \end{enumerate}
\end{lem}

\begin{prop}\label{prop:basis_1}Let $\cX$ be a representation-finite poset and $\cP$ be a finite poset. Recall that $\rI=\Ind\cX$ and fix $\rE$ a subset of $\Emb_{\cX}^{\cP}$. Assume that $\bigcup_{f\in \rE} \Im f=\cP$ and that the Gabriel quiver of the endomorphism algebra $\End_{\cP}(\bigoplus_{(f,U)\in \rE\times \rI/\sim_!}f_!U)$ has no oriented cycle. Then the set $ \cH^\rE:=\{f_!U \mid (f,U)\in \rE\times \rI/{\sim_!}\}$ is a $\bdh_{\cX,\cP}^{\rE}$-basis (and hence a $\bm_{\cX,\cP}^{\rE}$-basis). 
\end{prop}

\begin{proof}
The proof is a consequence of Theorem 4.22 in \cite{BBH}. We give an alternative proof here for the convenience of the reader.
For $(f,U)$ and $(g,V)$ in $\rE\times \rI/\sim_!$ :
$$\left(\bdh_{\cX,\cP}^{\rE}(f_!U)\right)_{(g,V)}= \dim \Hom_{\cP}(g_!V,f_!U).$$
Hence we have $\left(\bdh_{\cX,\cP}^{\rE}(f_!U)\right)_{(f,U)}=1$ and $\left(\bdh_{\cX,\cP}^{\rE}(f_!U)\right)_{(g,V)}\neq 0$ implies $(g,V)\leq (f,U)$ in the partial order given by the Gabriel quiver of $\End_{\cP}(\bigoplus_{(f,U)\in E\times I/\sim_!}f_!U)$. We conclude using Lemma \ref{lemma:basisPoset}.
\end{proof}

 \begin{prop}\label{prop:basis_2} Let $\cP$ be a connected poset, and $\cX$ be one of the following posets

 \[
  \scalebox{0.8}{
  \begin{tikzpicture}
  
\begin{scope}[xshift=4cm]

\node(x) at (0,0) {$\cX'_3=$};
\node (A) at (1,0) {$1$};
\node (B) at (2,0.5) {$2$};
\node (C) at (2,-0.5) {$3$};
\draw[->] (A)--(B);
\draw[->] (A)--(C);
\end{scope}

\begin{scope}[xshift=8cm]
\node(x) at (0,0) {$\cX''_3=$};
\node (A) at (1,0.5) {$1$};
\node (B) at (1,-0.5) {$2$};
\node (C) at (2,0) {$3$};
\draw[->] (A)--(C);
\draw[->] (B)--(C);
\end{scope}
 \end{tikzpicture}}\]
Let $\rE$ be a subset of $\Emb_{\cX}^{\cP}$ such that  $\bigcup_{f\in E} \Im f=\cP$. Then the set 
$\cH_\rE=\{ f_!U \mid (f,U)\in \rE\times \rI/\sim_!\}$ forms a $\bdh^\rE_{\cX,\cP}$-basis, and hence also a  $\bm^\rE_{\cX,\cP}$-basis. 

\end{prop}

\begin{proof}
For $\cX= \cX'_3$ we can use Proposition 5.2 of \cite{BBH} to see  that the objects $f_!U$ are ``single-source spreads'' since they have an indecomposable projective cover. Then it follows from Lemma 5.9 in \cite{BBH} that there is no oriented cycle in the Gabriel quiver of $\End_{\cP}(\bigoplus_{(f,U)\in \rE\times \rI/\sim_!}f_!U)$. 

For $\cX= \cX''_3$, we prove directly that there are no cycle in the quiver of $\End_{\cP}(\bigoplus_{(f,U)\in \rE\times \rI/\sim_!}f_!U)$. We show that if there are nonzero morphisms 
$$X_1\to X_2\to \cdots\to X_n\to X_1 $$ with $X_i$ of the form $f_!U$ for $(f,U)\in \rE\times \rI$ then $X_j\simeq X_1$ for any $j=1,\ldots,n$. Since the modules $f_!U$ are bricks this ends the proof. There are three types of modules of the form 
$f_!U$: the projective, the hooks and the modules of the form $\Coker(P_a\to P_b\oplus P_c)$ with $b$ and $c$ incomparable and $b,c<a$ that we call generalized hooks. First note that there is no nonzero morphism from a hook or a generalized hook to a projective. Since we already know the claim holds if each $X_i$ is projective, we may therefore assume that the $X_i$ are either hooks or generalized hooks. We assume here that all the modules $X_i= \Coker (P_{a_i}\to P_{b_i}\oplus P_{c_i})$ are generalized hooks, the proof easily generalizes when some of them are hooks. 

A nonzero morphism $X_i\to X_{i+1}$ induces a commutative square 
$$\xymatrix{P_{a_i}\ar[r]\ar[d] & P_{b_i}\oplus P_{c_i}\ar[d] \\ P_{a_{i+1} }\ar[r] & P_{b_{i+1}}\oplus P_{c_{i+1}}}$$
The second vertical map is nonzero, hence we may assume $b_i\geq b_{i+1}$. Then we obtain $a\geq b_{i+1}$ and by composing this map with the projection $P_{b_{i+1}}\oplus P_{c_{i+1}}\to P_{b_i+1}$ we obtain that the map $P_{a_i}\to P_{a_{i+1}}$ is nonzero which implies that $a_i\geq a_{i+1}$. Therefore we have $a_i=a_1$ for all $i=1,\ldots, n$. Now we obtain nonzero morphisms 
$$ P_{b_1}\oplus P_{c_1}\to P_{b_2}\oplus P_{c_2}\to \ldots \to P_{b_n}\oplus P_{c_n}\to P_{b_1}\oplus P_{c_1}.$$ This composition is nonzero since the composition of these maps with $P_{a_1}\to P_{b_1}\oplus P_{c_1}$ is nonzero. Therefore we have a chain $\alpha_1\geq \alpha_2\geq \ldots \geq \alpha_n\geq \alpha_1$ with $\alpha_i=b_i$ or $\alpha_i=c_i$. Furthermore in the commutative square above, it is easy to see that if $a_i=a_{i+1}$ and $b_i=b_{i+1}$, then we should have $c_i\geq c_{i+1}$, since $b_i$ and $c_i$ are incomparable by hypothesis. Hence we deduce for each $i=1,\ldots,n$ we have $\{b_i,c_i\}=\{b_1,c_1\}$. Thus all the $X_i$'s are isomorphic. This ends the proof.
\end{proof}


\section{Interval bases via families of embeddings}\label{sec:families}

 In this section, we first reinterpret classical bases for invariants, such as the simple modules for the dimension vector and the rectangle modules for the rank invariant, in terms of the intermediate extension functor. This provides a conceptual explanation for their appearance.

We then develop a more general framework based on families of embeddings of representation-finite posets, which allows us to systematically construct bases for invariants while avoiding redundancy. In particular, by restricting to suitable indecomposable modules associated to each embedding, we obtain bases consisting entirely of interval modules. This approach leads to explicit and computable descriptions of invariants, as well as new bases generalizing the rectangle modules.

\subsection{Interval modules and the intermediate extension functor}
It is well known that the simple modules form a basis for the dimension vector, and that the rectangle modules form a basis for the rank invariant. While the latter was established in  \cite{BOO} by an explicit computation, the underlying conceptual reason for the appearance of rectangle modules has remained less clear.

Our aim is to provide a unified explanation of these phenomena using the intermediate extension functor introduced in \cite{AET}. This functor allows us to associate to each embedding $f: \cX \rightarrow \cP$ and each indecomposable $\cX$-module $U$ a canonical $\cP$-module $\Theta_fU$, which behaves well with respect to interval structures. In particular, we show that both simple modules and rectangle modules arise naturally from this construction, corresponding to embeddings of the posets ${\mathcal X}_1$
 and ${\mathcal X}_2$, respectively.

We begin with an example illustrating that the induction and coinduction functors generally do not send interval modules to interval modules.

\begin{ex}\label{ex:not_interval}
    Consider the posets
    $$
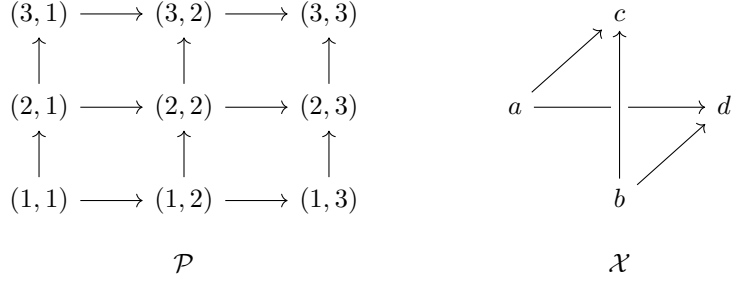
\begin{tikzcd}
        (3,1) \arrow[r] & (3,2) \arrow[r] & (3,3)&&& c&\\
        (2,1) \arrow[r]\arrow[u] & (2,2) \arrow[r]\arrow[u] & (2,3)\arrow[u]&&a\arrow[rr]\arrow[ur]&&d&\\
        (1,1) \arrow[r]\arrow[u] & (1,2) \arrow[r]\arrow[u] & (1,3)\arrow[u]&&&b\arrow[uu,crossing over]\arrow[ur]&\\[-1em]
        &\cP&&&&\cX
    \end{tikzcd}$$
    Let $f: \cX \rightarrow \cP$ be the poset embedding given by $fa = (2,1)$, $fb = (1,2)$, $fc = (3,2)$, and $fd = (2,3)$. Consider the sincere interval module $\mathbb{I}_{[\cX]} \in \mods \cX$. Since the induction functor $f_!$ preserves projective presentations and the coinduction functor $f_*$ preserves injective copresentations, we then have
    $$\begin{tikzcd}
        K \arrow[r,"1"] & K \arrow[r] & 0&& K \arrow[r,"1"] & K\arrow[r] & 0\\
        K \arrow[r,"\iota_1"]\arrow[u,"1"] & K^2 \arrow[r,"\nabla"]\arrow[u,"\nabla"] & K\arrow[u]&& K \arrow[u,"1"]\arrow[r,"\Delta"] & K^2\arrow[r,"\mathrm{pr}_1"]\arrow[u,"\mathrm{pr}_2"] & K\arrow[u]\\
        0 \arrow[r]\arrow[u] & K \arrow[r,"1"]\arrow[u,"\iota_2"] & K\arrow[u,"1"]&&0\arrow[u]\arrow[r] & K \arrow[u,"\Delta"] \arrow[r,"1"] & K\arrow[u,"1"]\\[-1em]
        &f_!\mathbb{I}_{[\cX]}&&&&f_*\mathbb{I}_{[\cX]}
    \end{tikzcd}$$
    where $\iota_i$ denotes the $i$-th inclusion map, $\mathrm{pr}_i$ denotes the $i$-th projection map, and $\Delta$ and $\nabla$ denote the diagonal and co-diagonal maps, respectively. We note that both $f_!\mathbb{I}_{[\cX]}$ and $f_*\mathbb{I}_{[\cX]}$ are indecomposable but are not interval modules. Moreover, we have that $\Hom_\cP(f_!\mathbb{I}_{[\cX]},f_*\mathbb{I}_{[\cX]}) \cong K$ and that the image of any nonzero map $f_!\mathbb{I}_{[\cX]} \rightarrow f_*\mathbb{I}_{[\cX]}$ is precisely $\mathbb{I}_S$ for $S = \cP \backslash \{ (1,1), (3,3) \}$.
\end{ex}

It is shown in \cite[Proposition~4.16]{AET} that the final observation in Example~\ref{ex:not_interval} is not a coincidence. In order to state the general result from \cite{AET}, we need the following definitions.

\begin{defn}
    Let $S \subseteq \cP$. The \emph{convex hull} of $S$ in $\cP$ is the set
    $$\overline{S} = \{x \in \cP \mid \exists y, z \in S \mid y \leq x \leq z\}.$$
\end{defn}

\begin{defn}
    Let $f: \cX \rightarrow \cP$ be a poset embedding and let $U \in \mods \cX$. By adjunction and the fact that $f^*f_* = \mathrm{Id}$, there is a natural isomorphism
    $\Hom_\cP(f_! U, f_*U) \cong \End_{\cX}(U)$. Let $\theta_fU: f_! U \rightarrow f_*U$ denote the image of the identity map $1_U \in \End_{\cX}(U)$ under this isomorphism. We denote $\Theta_fU = \mathrm{Im}(\theta_fU)$.
\end{defn}

\begin{rem}
    \begin{enumerate}
        \item As discussed in \cite[Section~4.2]{AET}, $\Theta_f$ is the so-called ``intermediate extension functor''. We refer to \cite{AET} for further background information and historical context. Note that we will not use the functoriality of $\Theta_f$ in this paper.
        \item Suppose that $U \in \mods \cX$ is an interval module. Then $\dim_K \End_\cX(U) \cong K$ by 
        \cite[Proposition~14]{DX} or \cite[Proposition~5.5]{BBH}. Thus every morphism $f_!U \rightarrow f_* U$ is a scalar multiple of $\theta_fU$ in this case.
    \end{enumerate}
\end{rem}

We now recall the following, which shows that the functor $\Theta_f$ preserves interval modules.

\begin{prop}\cite[Proposition~4.16]{AET}\label{prop:interval_functor}
    Let $f: \cX \rightarrow \cP$ be a poset embedding and $\mathbb{I}_S \in \mods \cX$ an interval module. Then $\Theta_f(\mathbb{I}_S) \cong \mathbb{I}_{\overline{f(S)}}$.
\end{prop}

To make this perspective more concrete, we now examine how the intermediate extension functor behaves in low-dimensional cases. In particular, we compare it with the induction and coinduction functors associated to embeddings $f: \cX \rightarrow \cP$.

\begin{ex}\label{ex:hooks_rects}
\begin{enumerate}
\item 
    We consider embeddings $f: \cX_2 \rightarrow \cP$, 
   where $\cX_2 = \{1\rightarrow 2\}$. Recall from Example~\ref{Ex::mult2} that the three indecomposable $\cX_2$-modules are $\bI_{\{1,2\}} = P_1 = I_2$, $\bI_{\{1\}} = H_{1,2} = I_a$, and $\bI_{\{2\}} = P_2 = C_{1,2}$. For $f \in \Emb_\cX^\cP$, one computes

    \begin{multicols}{3}
   \mbox{}
        \vspace{-30pt}
    \begin{eqnarray*}
        f_!\mathbb{I}_{\{1\}} &=& H_{f1,f2},\\
        f_*\mathbb{I}_{\{1\}} &=& I_{f1},\\
        \Theta_f\mathbb{I}_{\{1\}} &=& \mathbb{I}_{\{f1\}},
    \end{eqnarray*}
    \columnbreak
    \begin{eqnarray*}
        f_!\mathbb{I}_{\{2\}} &=& P_{f2},\\
        f_*\mathbb{I}_{\{2\}} &=& C_{f1,f2},\\
        \Theta_f\mathbb{I}_{\{2\}} &=& \mathbb{I}_{\{f2\}},\\
    \end{eqnarray*}
    \columnbreak
    \begin{eqnarray*}
        f_!\mathbb{I}_{\{1,2\}} &=& P_{f1},\\
        f_*\mathbb{I}_{\{1,2\}} &=& I_{f2},\\
        \Theta_f\mathbb{I}_{\{1,2\}} &=& \mathbb{I}_{[f1,f2]}.
    \end{eqnarray*}
    \end{multicols}
    \vspace{-30pt}
    \noindent In particular, we obtain that $\{\Theta_fU \mid f \in \Emb_\cX^\cP, U \in \Ind \cX\}$ is the set of rectangle modules. Note that there is some redundancy in this collection: for $x \in \cP$ we obtain the (simple) rectangle module $\mathbb{I}_{\{x\}}$ once for every embedding $f: \cX_2 \rightarrow \cP$ with either $x = f(1)$ or $x = f(2)$.

    It follows from \cite[Corollary~2.7]{BOO} that $\{\Theta_fU \mid f \in \Emb_\cX^\cP, U \in \Ind \cX\}$ is a $\brk_\cP$-basis, hence also a $\bm_{\cX_2,\cP}$ and $\bdh_{\cX_2,\cP}$-basis. 
    
    Note that we can define an equivalence relation $\sim_{\Theta}$ on $\rE\times \rI$ by setting $$(f,U)\sim_{\Theta} (g,V) \quad \Leftrightarrow \quad \Theta_f(U)=\Theta_g(V),$$ but the invariant $\bm_{\cX_2,\cP}$ does not factor through $\bZ^{\rE\times \rI/\sim_{\Theta}}$. For example, if $\cP=\{a\to b\to c\}$, then we have $\Theta_{f}\mathbb I_{\{1\}}=\mathbb I_{\{a\}}=\Theta_g\mathbb I_{\{1\}}$ for $\Im f=\{a,b\}$ and $\Im g=\{a,c\}$. However, we have for any $\cP$-module $M$ with $\rk M_{a\to b}\neq \rk M_{a\to c}$ that
    $$\bm_{\cX_2,\cP}(M)_{(f,I_{\{1\}})}=\dim M_a-\rk M_{a\to b}\neq \dim M_a-\rk M_{a\to c}=\bm_{\cX_2,\cP}(M)_{(g,I_{\{1\}})}.$$
\item Now we consider $\cX= \cX'_3$ or $\cX''_3$ as follows 

\[
  \scalebox{0.8}{
  \begin{tikzpicture}
  
\begin{scope}[xshift=4cm]

\node(x) at (0,0) {$\cX'_3=$};
\node (A) at (1,0) {$1$};
\node (B) at (2,0.5) {$2$};
\node (C) at (2,-0.5) {$3$};
\draw[->] (A)--(B);
\draw[->] (A)--(C);
\end{scope}

\begin{scope}[xshift=8cm]
\node(x) at (0,0) {$\cX''_3=$};
\node (A) at (1,0.5) {$1$};
\node (B) at (1,-0.5) {$2$};
\node (C) at (2,0) {$3$};
\draw[->] (A)--(C);
\draw[->] (B)--(C);
\end{scope}
 \end{tikzpicture}}\] For embeddings $f: \cX \rightarrow \cP$, the intermediate extension functor yields interval modules that can be viewed as replacements for the more complicated modules obtained via induction or coinduction (see Example \ref{ex::f!U modules}). In particular, this provides a systematic way to avoid non-interval modules and obtain simpler building blocks for the invariant.

 We recall the description of the indecomposable $\cX_3'$- and $\cX_3''$-modules from the proof of Proposition~\ref{prop:X_3Rank}and Example~\ref{ex::f!U modules}:
the modules $\Theta_f(U)$ are either simples, rectangles, or interval modules 
of the form ${\mathbb I}_{[f(1),f(2)] \cup [f(1),f(3)]}$
when $\cX=\cX'_3$, and ${\mathbb I}_{[f(1),f(3)] \cup [f(2),f(3)]}$
when $\cX=\cX''_3$. In the case of a 2D-grid poset, these interval modules look as follows: 

 \[
  \scalebox{0.8}{
  \begin{tikzpicture}

\draw (0,0)--(6,0)--(6,4)--(0,4)--(0,0);

\draw[fill=blue!30] (1,1)--(4,1)--(4,2)--(2,2)--(2,3)--(1,3)--(1,1);

\node at (1,1) {$\bullet$};
\node at (0.8,0.8) {$f(1)$};
\node at (4,2) {$\bullet$};
\node at (4.2,2.2) {$f(2)$};
\node at (2,3) {$\bullet$};
\node at (2.2,3.2) {$f(3)$};
\node at (1.6,1.6) {$k$};
\node at (2,-0.5) {$\cX=\cX'_3$};

\begin{scope}[xshift=8cm]
\draw (0,0)--(6,0)--(6,4)--(0,4)--(0,0);

\draw[fill=blue!30] (1,2)--(2,2)--(2,1)--(4,1)--(4,3)--(1,3)--(1,2);

\node at (4.2,3.2) {$f(3)$};
\node at (0.6,2) {$f(2)$};
\node at (2,0.8) {$f(1)$};
\node at (1,2) {$\bullet$};
\node at (2,1) {$\bullet$};
\node at (4,3) {$\bullet$};
\node at (3,2) {$k$};
\node at (2,-0.5) {$\cX=\cX''_3$};
\end{scope}

\end{tikzpicture}}\]

\end{enumerate}
    
\end{ex}

\subsection{Families of embeddings}

In Example \ref{ex:hooks_rects}, we considered order-embeddings of {\em one fixed} poset $\cX$ into the poset $\cP$, and studied the set of 
{ \em all} 
indecomposables obtained from $\cX$. That turns out to introduce a lot of redundancy and some difficulty to describe the image of the invariant. 
 To avoid redundancy,  
we study now embeddings of several posets of increasing sizes, but limit attention to only some indecomposables (that have not been obtained from embedding of smaller posets previously). One of the advantage here is that we do not need to assume $\cX$ to be of finite representation type.

\begin{defn}
Let $\cP$ be a poset. A \emph{family of $\cP$-order embeddings} is a finite set $\cF:=\{(\cX_i,\rE_i)\}_i$ with each $\cX_i$ a finite poset (not necessarily representation finite) and each $\rE_i$ a subset of $\Emb_{\cX_i}^{\cP}$.  
For $\cF=\{(\cX_i,\rE_i)\}_i$ a family of $\cP$-order embeddings, we denote by $\rE:=\bigsqcup_{i}\rE_i$ and we define an invariant $\bm_{\cF}:\Ksp (\cP)\to \bZ^\rE$ by 
$$ \left(\bm_{\cF}(M)\right)_f:= {\rm mult} (\mathbb I_{\cX_i},f^*(M)) \quad \textrm{for }f\in \rE_i. $$
\end{defn}

 \begin{ex}\label{ex:iterated_mult}
\begin{enumerate}
    \item Let $\cF:=\{(\cX_1,\Emb_{\cX_1}^\cP),(\cX_2,\Emb_{\cX_2}^\cP)\}$ for $\cX_1$ and $\cX_2$ as in Example~\ref{Ex::mult2}. We show in Example \ref{Example5.13} that $\rE=\Emb_{\cX_1}^\cP\sqcup \Emb_{\cX_2}^\cP$ is in natural bijection with $\{\leq_{\cP}\}$ and via this bijection we have $\bm_{\cF}=\brk_{\cP}$.
    \item  The generalized rank invariant of \cite{KM} and the various compressed multiplicities of \cite{AENY2} can both be realized in the form $\bm_\cF$. More precisely:
    \begin{enumerate}
        \item Let $\mathfrak{I}$ be the set of posets which embed into $\cP$ as intervals, and for each $\cX_i \in \mathfrak{I}$ let $\rE_i \subseteq \Emb_{\cX_i}^\cP$ be the set of embeddings with interval image. Then, as a result of \cite[Lemma~3.1]{CL}, the generalized rank invariant of \cite{KM} is equal to $\bm_\cF$ for $\cF=\{(\cX_i,\rE_i)\mid \cX_i \in \mathfrak{I}\}$. 
        \item Suppose that $\cP$ is the product of two (finite) totally ordered sets. Let $\mathfrak{J}$ be the set of connected posets which do not contain any chains of length $\geq 3$ and can be embedded into $\cP$. If one computes the ``sink-source compression factor'' of \cite{AENY2} with respect to every interval subset of $\cP$, the result is precisely $\bm_\cF$ for $\cF={\{(\cX_i,\Emb_{\cX_i}^\cP)\mid \cX_i \in \mathfrak{J}\}}$. A similar result likewise holds for the ``corner-complete compression factors'', while the ``total compression factors'' coincide with the generalized rank invariant.
    \end{enumerate}
    See also \cite[Section~3.2]{BBH} and \cite[Example~4.14]{AENY2} for additional discussion.
\end{enumerate}
 
 \end{ex}

\begin{prop}

Let $\cP$ be a finite poset and $\cF:=\{(\cX_i,\rE_i)\}_i$ be a family of $\cP$-order embeddings.  We define an equivalence relation on $\rE = \bigsqcup_{i}\rE_i$ by $f\sim g$ if $\overline{\Im f}=\overline{\Im g}$. Then the invariant $\bm_{\cF}:\Ksp(\cP)\to \bZ^{\rE}$ induces an isomorphism of abelian groups
$$\langle [\mathbb I_{\overline{\Im f}}], \ f\in \rE\rangle \overset{\sim}{\longrightarrow} \bZ^{\rE/\sim}.$$
\end{prop}

\begin{proof}

For $f\in \rE_i$ and $g \in \rE_j$, we have 
$$\left(\bm_{\cF}(\mathbb I_{\overline{\Im f}})\right)_g= {\rm mult}_{\cX_j}(\mathbb I_{\cX_j}, g^{*}\mathbb I_{\overline{\Im f}}).$$
The $\cX_j$-modules $\mathbb I_{\cX_j}$ and $ g^{*}\mathbb I_{\overline{\Im f}}=\mathbb{I}_{g^{-1}(\overline{\Im f})}$ are both interval modules hence indecomposable. Therefore we obtain $\left(\bm_{\cF}(\mathbb I_{\overline{\Im f}})\right)_g$ is $0$ or $1$ and 
$$\begin{array}{rcl}
\left(\bm_{\cF}(\mathbb I_{\overline{\Im f}})\right)_g= 1 \ & \Leftrightarrow  & \bI_{\cX_j}= \bI_{g^{-1}(\overline{\Im f})}\\ & \Leftrightarrow & \cX_j=g^{-1}(\overline{\Im f})\\  & \Leftrightarrow & \cX_j\subseteq g^{-1}(\overline{\Im f})\\ & \Leftrightarrow & \Im g\subseteq \overline{\Im f} \\ & \Leftrightarrow & \overline{\Im g}\subseteq \overline{\Im f}.\end{array}$$
Then by defining an order relation $\leq$ on $\rE/\sim$ by 
$$g\leq f \quad \Leftrightarrow \quad \overline{\Im g}\subseteq \overline{\Im f},$$
we conclude using Lemma \ref{lemma:basisPoset}.
\end{proof}

    It is not clear in general that the invariant $\bm_{\cF}$ factors through the injection $\bZ^{\rE/\sim}\to \bZ^\rE$. However, it can be shown in some cases.

    \begin{ex}\label{ex:bases_mobius}
    \begin{enumerate}
        \item Consider the setup of Example~\ref{ex:iterated_mult}(1). Then $\{\mathbb{I}_{\overline{\Im f}} \mid f \in \rE\}$ is the set of rectangle modules. As previously mentioned, it is shown in \cite[Corollary~2.7]{BOO} that this forms a basis of $\bm_\cF$.
        \item Consider the setup of Example~\ref{ex:iterated_mult}(2a) or (2b). In both cases, $\{\mathbb{I}_{\overline{\Im f}} \mid f \in \rE\}$ is the set of interval modules. Both papers show that these form a basis for the corresponding invariant $\bm_\cF$ using M\"obius inversion arguments. (This is also established for the ``corner-complete compression factors'' in \cite{AENY2}.)
    \end{enumerate}
    \end{ex}

    We now show that the perspective of iterated embeddings, combined with the intermediate extension functor, yields explicit bases for the associated invariants. In particular, these bases consist of interval modules and generalize the classical rectangle-module basis for the rank invariant.

\begin{thm}\label{thm::IntervalBasis}
        Let $\cP$ be a finite poset, and let $\cF=\{(\cX_i,\rE_i)\}$ be a family of $\cP$-order embeddings such that there is no chain of length $\geq 3$ in  $\cX_i$ for each $i$. Denote $\rE = \bigsqcup_{i}\rE_i$. Then the invariant $\bm_\cF$ admits a basis 
        $$\cR_\rE:= \{\mathbb I_{\overline{\Im f}} \mid f\in \rE\}\subseteq \Ksp(\cP)$$ 
        given by the modules obtained via the intermediate extension functor from the characteristic indecomposable modules of the $\cX_i$. 
\end{thm}

\begin{proof}
We show that in this case the equivalence relation on $\rE$ is trivial and conclude by the previous proposition. Indeed let $f\in \rE_i$ and $g\in \rE_j$ such that $f\sim g$, that is $\overline{\Im f}=\overline{\Im g}$. Then let $x\in \cX_i$. By hypothesis $x$ is either minimal or maximal. Assume first it is minimal.  Then $f(x)\in \Im f\subseteq \overline{\Im f}=\overline{\Im g}$, so  there exists $y,y'\in \cX_j$ with 
$g(y)\leq f(x)\leq g(y')$. Since $g(y)$ is in $\Im g$, it is in the convex hull of $\Im f$, but since $x$ is minimal we have $f(x)=g(y)$. By a similar argument, if $x$ is maximal, then $f(x)$ is in $\Im g$. Therefore we obtain $\Im f\subseteq \Im g$, and by symmetry we get $\Im f=\Im g$, that is $f=g$. 
\end{proof}

\begin{rem}
It is an interesting problem to find a representation-theoretic argument generalizing the basis result for the generalized rank invariant established using M\"obius inversion in \cite{KM}. In particular, note that the ``zigzags'' used in \cite{DKM} to compute the generalized rank invariant admit injective poset morphisms into $\cP$, but that these morphisms are generally not full as functors between poset categories, and thus not embeddings. It would be interesting to determine the extent to which the theory developed in this paper could be adapted to such a setting.
\end{rem}

\subsection{Comparing $\bm$-invariants}

Let $\cP$ be a finite poset, $\cX$ be a representation-finite poset and $\rE$ be a subset of $\Emb_\cX^\cP$. As mentionned before, in general the rank of the invariant $\bm_{\cX,\cP}^\rE$ is much smaller than $|\rE\times \Ind\cX|$. The aim of this section is to find a family $\cF=\{(\cX_i,\rE_i)_i\}$ of $\cP$-order embeddings with $\sum_i|\rE_i|\leq |\rE\times \Ind\cX|$ and with $\bm_{\cF}\sim \bm_{\cX,\cP}^\rE$. We start with the example coming from \ref{rem::multX2_redundant} to illustrate the strategy.

\begin{ex}\label{Example5.13}
Let $\cP$ be a finite poset and $\cF=\{(\cX_1,\Emb_{\cX_1}^{\cP}),(\cX_2,\Emb_{\cX_2}^{\cP})\}$ given as follows.

\[
  \scalebox{0.8}{
  \begin{tikzpicture}
  
\begin{scope}[xshift=-3cm]
 \node(X) at (0,0) {$\cX_1=$};
  \node (B) at (1,0) {$1$};
\end{scope}

\begin{scope}[xshift=0cm]
   \node(X) at (0,0) {$\cX_2=$};
  \node (A) at (1,0) {$1$};
  
\node (B) at (3,0) {$2$};

\draw[->] (A)--(B);
\end{scope} 

 \end{tikzpicture}}\]

One can define a bijection $\{\leq_{\cP}\}\to \Emb_{\cX_1}^\cP\sqcup \Emb_{\cX_2}^\cP$ sending $(a,a)\in \{\leq_\cP\}$ to $f_a\in \Emb_{\cX_1}^\cP$ with $f_a(1)=a$ and $(a,b)$ to $f_{ab}:\cX_2\to \cP$ with $f_{ab}(1)=a$ and $f_{ab}(2)=b$. Then it is straightforward to check that we have a commutative square
$$\xymatrix{\Ksp(\cP)\ar[rr]^-{\bm_{\cF}}\ar@{=}[d] & & \bZ^{\rE_1\sqcup \rE_2}\ar[d]^\sim \\ \Ksp(\cP)\ar[rr]^-{\brk_\cP} & & \bZ^{\{\leq_\cP\}}}$$
showing that $\bm_{\cF}\sim \brk_\cP$. Combining this with Corollary \ref{corX_2Rank}, we obtain that if $\cP$ is connected with at least $2$ elements then we have an equivalence 
$$\bm_{\cX_2,\cP}\sim \bm_{\cF}.$$
As noted in Remark \ref{rem::multX2_redundant},
computing $\bm_{\cF}$ is roughly three times more efficient than computing $\bm_{\cX_2,\cP}.$

This equivalence can also be seen directly using the following arguments.  Consider the family of $\cX_2$-order embeddings $$\cF_{\cX_2}:=((\cX_1,\Emb_{\cX_1}^{\cX_2}),(\cX_2,\Emb_{\cX_2}^{\cX_2})).$$
Then we have $\Emb_{\cX_1}^{\cX_2}=\{f_1,f_2\}$ with $\Im f_x=\{x\}$, and  $\Emb_{\cX_2}^{\cX_2}=\{{\rm id}\}$. We immediately check for $M$ in $\mod \cX_2$
$$ \left(\bm_{\cF_{\cX_2}}(M)\right)_f  = \left\{ \begin{array}{cl} \dim M_x & \textrm{ if }f=f_x \\ \rk M_{1\to 2} & \textrm{ if }f={\rm id} \end{array}\right.$$
Then one has a commutative square

$$\xymatrix{\Ksp(\cX_2)\ar[rr]^-{\bm_\cX}\ar@{=}[dd] & & \bZ^{\{\bI_{\{2\}},\bI_{\{1,2\}},\bI_{\{1\}}\}}\\ & & \\ \Ksp(\cX_2)\ar[rr]^-{\bm_{\cF_\cX}} & & \bZ^{\{f_1,f_2,{\rm id}\}}\ar[uu]_{\Phi=\begin{pmatrix}0 & 1 & -1\\ 0 & 0 & \;\;\;1\\ 1 & 0 & -1\end{pmatrix}}}$$
Now fix a $g\in \Emb_{\cX_1}^\cP$. Since $\cP$ is connected and has at least $2$ elements, this embedding $g$ can be written as $g=f\circ f_{x}$ with $f\in \Emb_{\cX_2}^\cP$ and $f_x\in \Emb_{\cX_1}^{\cX_2}$ ($x=1,2$). Then one can write \begin{align*} (\bm_{\cF}(M))_g  &= (\bm_{\cX_1}(g^*M))_{\bI_{\{1\}}}\\  
& =  (\bm_{\cX_1}((f\circ f_x)^*M))_{\bI_{\{1\}}}\\ 
& = (\bm_{\cX_1}(f_x^*(f^*M)))_{\bI_{\{1\}}}\\ & =  (\bm_{\cF_{\cX_2}}(f^*M))_{f_x}\\ & =  \left(\Phi^{-1}\circ \bm_\cX(f^*M)\right)_{f_x} \\ & =  \left(\Phi^{-1}\circ \bm_{\cX,\cP}(M)_{(f,-)}\right)_{f_x}.\end{align*}
A similar argument can be used if $g\in \Emb_{\cX_2}^\cP$ which gives a direct proof that the invariants $\bm_{\cX_2,\cP}$ and $\bm_\cF$ are equivalent.
\end{ex}

Using the same kind of arguments, one can show the following more general result.

\begin{prop}\label{prop::equivalenceX-FX}

Let $\cP$ be a finite poset, $\cX$ be a representation-finite poset and $\rE$ be a subset of $\Emb_\cX^\cP$. Let $\cF_\cP=\{(\cX_i,\rE_i)_i\}$ be a family of $\cP$-order embeddings. We denote by $\cF_{\cX}=\{(\cX_i,\Emb_{\cX_i}^{\cX})\}_i$ the family of $\cX$-order embeddings of the posets $\cX_i$.
    \begin{enumerate}
    \item If for any $(\cX_i,\rE_i) \in \cF$ the composition map $\Emb_{\cX_i}^\cX\times \Emb_{\cX}^\cP\rightarrow \Emb_{\cX_i}^\cP$ restricts to a map $\Emb_{\cX_i}^\cX\times \rE\rightarrow \rE_i$, then 
    $$\bm_{\cF_{\cX}}\geq \bm_{\cX}\quad \Rightarrow\quad \bm_{\cF_\cP}\geq \bm_{\cX,\cP}^{\rE}.$$ 
    \item If moreover the restriction $\Emb_{\cX_i}^\cX\times \rE\rightarrow \rE_i$ is surjective, then 
    $$\bm_{\cF_{\cX}}\simeq \bm_{\cX}\quad \Rightarrow\quad \bm_{\cF_\cP}\simeq\bm_{\cX,\cP}^{\rE}.$$
    \end{enumerate}
\end{prop}

\begin{proof}
(1) By hypothesis there exist integers $\alpha_{(g,U)}$ for any $g\in \bigsqcup_i \Emb_{\cX_i}^\cX$ and $U\in \Ind (\cX)$ such that for any $M\in \mod \cX$
\[\left(\bm_\cX(M)\right)_U = \sum_{g\in \bigsqcup_i\Emb_{\cX_i}^\cX} \alpha_{(g,U)}\left(\bm_{\cF_\cX}(M)\right)_{g}, 
\] that is 
\[ {\rm mult}_{\cX}(U,M) = \sum_{g\in \bigsqcup_i\Emb_{\cX_i}^\cX} \alpha_{(g,U)}{\rm mult}_{\cX_i}(\mathbb I_{\cX_i},g^* M).
\]

Now for $(f,U)\in \rE\times \Ind(\cX)$,  and $M\in \mod \cP$ we have 
$$\begin{array}{rcl}
\left(\bm_{\cX,\cP}^\rE(M)\right)_{(f,U)} & = & {\rm mult}_{\cX}(U,f^*M) \\ 
& = & \sum_{g\in \bigsqcup_i \Emb_{\cX_i}^\cX} \alpha_{(g,U)}{\rm mult}_{\cX_i}(\mathbb I_{\cX_i}, g^*(f^*M))\\
 & = & \sum_{g\in \bigsqcup_i\Emb_{\cX_i}^\cX} \alpha_{(g,U)}{\rm mult}_{\cX_i}(\mathbb I_{\cX_i}, (f\circ g)^* M)) \\ & = & \sum_{h\in \bigsqcup_i \rE_i} \beta_{(h,f,U)} \left(\bm_{\cF_{\cP}}(M)\right)_h \end{array}$$
 where $\beta_{(h,f,U)}=\alpha_{(g,U)}$when $h=f\circ g$. Since $f$ is an embedding the value of $\beta_{(h,f,U)}$ is well-defined.

 (2) Now assume that there exist integers $\alpha'_{(g,U)}$ for $g\in\bigsqcup_i \Emb_{\cX_i}^\cX$ and $U\in \Ind \cX$ such that for any $M\in \mod \cX$
\[\left(\bm_{\cF_{\cX}}(M)\right)_g  = \sum_{U\in \Ind\cX} \left(\bm_{\cX}(M)\right)_U, \]
that is 
\[ {\rm mult}_{\cX_i}(\mathbb I_{\cX_i},g^*M) =  \sum_{U\in \Ind(\cX)} \alpha'_{(g,U)} {\rm mult}_{\cX}(U,M).\]

Now for $h\in \rE_i$ fix $h_1\in \Emb_{\cX_i}^\cX$ and $h_2\in \rE$ such that $h=h_2\circ h_1 $. Then we have 

\[\begin{array}{rcl}
\left(\bm_{\cF_\cP}(M)\right)_h & =   & {\rm mult}_{\cX_i}(\mathbb I_{\cX_i},h^*M)  \\
 & = & {\rm mult}_{\cX_i}(\mathbb I_{\cX_i},h_1^*(h_2^*M)) \\ 
 & = & \sum_{U\in \Ind\cX} \alpha'_{(h_1,U)}{\rm mult}_\cX (U,h_2^*M) \\
     & = &\sum_{f\in \rE}\sum_{U\in \Ind\cX} \beta'_{(h,f,U)} \left(\bm_{\cX,\cP}^\rE (M)\right)_{(f,U)}
\end{array}\]
where $\beta'_{(h,f,U)}=\alpha'_{(h_1,U)}$ when $f=h_2$. 
\end{proof}

We conclude by collecting several consequences of Proposition~\ref{prop::equivalenceX-FX}.

\begin{prop}

Let $\cP$ be a finite connected poset and consider the following families of order embeddings $$\cF:=((\cX_1,\Emb_{\cX_1}^\cP),(\cX_2,\Emb_{\cX_2}^\cP), (\cX'_3,\Emb_{\cX'_3}^\cP)), \quad \left(\textrm{resp.} \ \cF:=((\cX_1,\Emb_{\cX_1}^\cP),(\cX_2,\Emb_{\cX_2}^\cP), (\cX''_3,\Emb_{\cX''_3}^\cP))\right),$$  where :
\[
  \scalebox{0.8}{
  \begin{tikzpicture}  
\begin{scope}[xshift=-3cm]
  \node (X) at (0,0) {$\cX_1=$};
  \node (B) at (1,0) {$1$};
\end{scope}

\begin{scope}[xshift=0cm]
   \node (X) at (0,0) {$\cX_2=$};
  \node (A) at (1,0) {$1$}; 
\node (B) at (2,0) {$2$};
\draw[->] (A)--(B);
\end{scope}

\begin{scope}[xshift=4cm]
\node (x) at (0,0) {$\cX'_3=$};
\node (A) at (1,0) {$1$};
\node (B) at (2,0.5) {$2$};
\node (C) at (2,-0.5) {$3$};
\draw[->] (A)--(B);
\draw[->] (A)--(C);
\end{scope}

\begin{scope}[xshift=8cm]
\node(x) at (0,0) {$\cX''_3=$};
\node (A) at (1,0.5) {$1$};
\node (B) at (1,-0.5) {$2$};
\node (C) at (2,0) {$3$};
\draw[->] (A)--(C);
\draw[->] (B)--(C);
\end{scope}

\end{tikzpicture}} 
\]
 For $\cX=\cX'_3$ (resp. $\cX=\cX''_3$) assume the following  
$$(*)\quad \textrm{for any $a<b$ in $\cP$ there exists $f\in \Emb^\cP_{\cX}$ with $\{a,b\}\in \Im f$}$$ 
 
Then we have an equivalence between the invariants $\bm_{\cX,\cP}$ and $\bm_{\cF'}$. 
\end{prop}

\begin{proof}We prove the statement for $\cX=\cX'_3$; the proof is similar for $\cX=\cX''_3$.
We consider the family $$\cF_\cX=((\cX_1,\Emb_{\cX_1}^\cX),\cX_2,\Emb_{\cX_2}^\cX),\cX'_3,\Emb_{\cX'_3}^\cX))$$ and show the equivalence between $\bm_{\cX}$ and $\bm_{\cF_\cX}$. 

We have $$\Emb_{\cX_1}^{\cX}=\{f_1,f_2,f_3\}, \quad \Emb_{\cX_2}^{\cX}=\{f_{12},f_{13}\}, \quad \textrm{ and }\Emb_{\cX'_3}^{\cX}=\{{\rm id}\} $$ where $f_a(1)=i$, $f_{ab}(1)=a$ and $f_{ab}(2)=b$.
Then one easily checks that we have a commutative square

$$\xymatrix{\Ksp(\cX)\ar[rr]^-{\bm_\cX}\ar@{=}[d] & & \bZ^{\{\bI_{\{1\}},\bI_{\{2\}},\bI_{\{3\}},\bI_{[1,2]},\bI_{[1,3]},\bI_{\cX'_3} \}}\\  \Ksp(\cX)\ar[rr]^-{\bm_{\cF_\cX}} & & \bZ^{\{f_1,f_2,f_3,f_{12},f_{13},{\rm id}\}}\ar[u]_{\Phi}}$$
with 
 
$$\Phi=\begin{pmatrix} 1 & 0 & 0 & -1 & -1 & 1 \\ 0 & 1 & 0 & -1 & 0 & 0\\ 0 & 0 & 1 & 0 & -1 & 0\\ 0 & 0 & 0 & 0  & 1  & -1\\ 0 & 0 & 0 & 1 & 0 & -1\\ 0 & 0 & 0 & 0 & 0 & 1 \end{pmatrix}.$$

Now condition $(*)$ and the connectedness of $\cP$ ensures that the composition map $\Emb_{\cX_i}^\cX\times \Emb_{\cX}^\cP\to \Emb_{\cX_i}^\cP$ is surjective for any $i$, and we can conclude using Proposition \ref{prop::equivalenceX-FX}  

\end{proof}

\begin{coro}\label{thm::eqX3-family} Let $\cP$ be a finite connected poset, and let $\cX$ be one of the following posets

\[
  \scalebox{0.8}{
  \begin{tikzpicture}

\begin{scope}[xshift=4cm]

\node(x) at (0,0) {$\cX'_3=$};
\node (A) at (1,0) {$1$};
\node (B) at (2,0.5) {$2$};
\node (C) at (2,-0.5) {$3$};
\draw[->] (A)--(B);
\draw[->] (A)--(C);
\end{scope}

\begin{scope}[xshift=8cm]
\node(x) at (0,0) {$\cX''_3=$};
\node (A) at (1,0.5) {$1$};
\node (B) at (1,-0.5) {$2$};
\node (C) at (2,0) {$3$};
\draw[->] (A)--(C);
\draw[->] (B)--(C);
\end{scope}
 \end{tikzpicture}}\]
 Assume the following  
$$(*)\quad \textrm{for any $a<b$ in $\cP$ there exists $f\in \Emb_\cX^\cP$ with $\{a,b\}\in \Im f$}$$ 
then the set 
$\cR_\rE = \{\Theta_f(U)\mid(f,U)\in \rE \times \rI\}$
forms a basis of $\bm_{\cX,\cP}$.
\end{coro}

\begin{proof}This is a direct consequence of the above proposition with Theorem \ref{thm::IntervalBasis}. Note that the basis
$$\{\bI_{\overline{\Im f}} \mid f\in \Emb_{\cX_i}^\cP, i=1,2,3\}$$ contains all simples, all rectangles, and all generalized rectangles as described in Example \ref{ex:hooks_rects} (2). 
One easily checks that the set 
$\cR_\rE$ contains exactly the simples, the rectangles and the generalized rectangles. The fact that it contains all the rectangles comes from condition $(*)$

\end{proof}

\begin{thm}\label{thm::eqX4-family} Let $\cP$ be a finite connected poset, and let $\cX$ be the following poset

\[
  \scalebox{0.8}{
  \begin{tikzpicture}
  \node at (-1,0) {$\cX=$};
  \node (A) at (0,0) {$a$};
  \node (B) at (1,0.5) {$b$};
  \node (C) at (1,-0.5) {$c$};
\node (D) at (2,0) {$d$};

\draw[->] (A)--(B);
\draw[->] (B)--(D);

\draw[->] (A)--(C);
\draw[->] (C)--(D);

 \end{tikzpicture}},\]
 and $\cX_1$, $\cX_2$, $\cX'_3$ and $\cX''_3$ as above.
Denote by $\rE:=\Emb_{\cX}^\cP$.
 Assume that the composition map $$ \Emb_{\mathcal{Y}}^\cX\times \rE\longrightarrow  \Emb_{\mathcal{\mathcal{Y}}}^\cP$$ is surjective  for $\mathcal{Y}=\cX_1$, $\cX_2$, $\cX'_3$ and $\cX''_3$.
Then the set 
$\cR_\rE = \{\Theta_f(U)\mid(f,U)\in \rE \times \Ind (\cX)\}$
forms a basis of $\bm_{\cX,\cP}$.
\end{thm}

\begin{proof}
We introduce the family of embeddings 
$$\cF=\left((\cX_1,\Emb_{\cX_1}^\cP), (\cX_2,\Emb_{\cX_2}^\cP),(\cX'_3,\Emb_{\cX'_3}^\cP),(\cX''_3,\Emb_{\cX''_3}^\cP)\right)$$ and 
$$\cF_\cX=\left((\cX_1,\rE_1:=\Emb_{\cX_1}^\cX), (\cX_2,\\rE_2:=\Emb_{\cX_2}^\cX),(\cX'_3,\rE'_3:=\Emb_{\cX'_3}^\cX),(\cX''_3,\rE''_3:=\Emb_{\cX''_3}^\cX)\right).$$ We first prove the equivalence between $\bm_{\cX}$ and $\bm_{\cF_\cX}$. Indeed one can check that $$|\rE_1|= 4, \ |\rE_2|=5, \ |\rE'_3|=1, \ |\rE''_3|=1$$ and construct an invertible map $\Phi:\bZ^{\rE_1\sqcup \rE_2\sqcup \rE'_3\sqcup \rE''_3}$ to $\bZ^{\Ind\cX}$ such that $\Phi\circ\bm_{\cF_\cX}=\bm_\cX$. Then we can apply Proposition \ref{prop::equivalenceX-FX} to obtain an equivalence $$\bm_{\cF}\sim \bm_{\cX,\cP}.$$
By Theorem \ref{thm::IntervalBasis} we obtain a basis of $\bm_{\cX,\cP}$ by considering $\cP$-modules of the form $\Theta_f(U)$ for $f\in \rE_1$ (resp. $\rE_2$, $\rE'3$, $\rE''_3$)  and $U=\bI_{\cX_1}$ (resp. $\bI_{\cX_2}$, $\bI_{\cX'_3}$, $\bI_{\cX''_3}$). We can then check that these modules correspond to $\Theta_f(U)$ where $f\in \rE$, and $U\in \Ind\cX$.

\end{proof}

We finish by giving more examples of families of embeddings recovering known invariants using Proposition \ref{prop::equivalenceX-FX}. 

\begin{ex}
\begin{enumerate}

\item 
Let $\cF$ be the family of posets

 \[
  \scalebox{0.8}{
  \begin{tikzpicture}
  
\begin{scope}[xshift=-3cm]
  \node(X) at (0,0) {$\cX_1=$};
  \node (B) at (1,0) {$1$};
\end{scope}

\begin{scope}[xshift=0cm]
   \node(X) at (0,0) {$\cX_2=$};
  \node (A) at (1,0) {$1$};
  
\node (B) at (2,0) {$2$};

\draw[->] (A)--(B);
\end{scope}

\begin{scope}[xshift=4cm]

\node(x) at (0,0) {$\cX'_3=$};
\node (A) at (1,0) {$1$};
\node (B) at (2,0.5) {$2$};
\node (C) at (2,-0.5) {$3$};
\draw[->] (A)--(B);
\draw[->] (A)--(C);
\end{scope}

\begin{scope}[xshift=9cm]
\node at (-1,0) {$\cX_4=$};
  \node (A) at (0,0) {$1$};
  \node (B) at (1,0) {$2$};
  \node (C) at (2,0.5) {$3$};
  
\node (D) at (2,-0.5) {$4$};

\draw[->] (A)--(B);
\draw[->] (B)--(C);
\draw[->] (B)--(D);
\end{scope}
 \end{tikzpicture}}\] 
and let $\cX:=\cX_4.$ Then one can check that $$|\Emb_{\cX_1}^{\cX}|= 4, \ |\Emb_{\cX_2}^{\cX}|=5, \ |\Emb_{\cX'_3}^{\cX}|=2, \ |\Emb_{\cX_4}^{\cX}|=1,$$
which gives rank $12$ for the invariant $\bm_{\cF_\cX}.$
The invariant $\bm_\cX$ is complete and also of rank $12$. On can even prove that we have an equivalence 

$$ \bm_\cX\sim \bm_{\cF_\cX}.$$

 Indeed one can observe  that $\bm_{\cF_\cX}$ determines all the interval indecomposable summands (using inclusion and exclusion principle). Then, this determination also uniquely determines the multiplicity of the only non-interval indecomposable summand, showing that $\bm_{\cF_\cX}$ is also complete.

\item Let $\cF$ be the following family of posets, and $\cX=\cX'_4$.
 \[
  \scalebox{0.8}{
  \begin{tikzpicture}
  
\begin{scope}[xshift=-3cm]
  \node(X) at (0,0) {$\cX_1=$};
  \node (B) at (1,0) {$1$};
\end{scope}

\begin{scope}[xshift=0cm]
   \node(X) at (0,0) {$\cX_2=$};
  \node (A) at (1,0) {$1$};
  
\node (B) at (2,0) {$2$};

\draw[->] (A)--(B);
\end{scope}

\begin{scope}[xshift=4cm]

\node(x) at (0,0) {$\cX'_3=$};
\node (A) at (1,0) {$1$};
\node (B) at (2,0.5) {$2$};
\node (C) at (2,-0.5) {$3$};
\draw[->] (A)--(B);
\draw[->] (A)--(C);
\end{scope}

\begin{scope}[xshift=9cm]
\node at (-1,0) {$\cX'_4=$};
\node (A) at (0,0) {$a$};
  \node (B) at (1,0.5) {$b$};
  \node (C) at (1,0) {$c$};
  
\node (D) at (1,-0.5) {$d$};

\draw[->] (A)--(B);
\draw[->] (A)--(C);
\draw[->] (A)--(D);

\end{scope}
 \end{tikzpicture}}\]

 Then one can check that 

$$|\Emb_{\cX_1}^{\cX}|= 4, \ |\Emb_{\cX_2}^{\cX}|=3, \ |\Emb_{\cX'_3}^{\cX}|=3, \ |\Emb_{\cX'_4}^{\cX}|=1.$$

Here the invariant $\bm_{\cX}$ is strictly finer than the invariant $\bm_{\cF_\cX}$.

 \end{enumerate}

\end{ex}

\end{document}